\documentclass[12pt]{article} %

\usepackage[utf8]{inputenc}
\usepackage{graphicx}
\usepackage{amsmath}
\usepackage{amssymb}
\usepackage{amsfonts}
\usepackage{amsthm}
\usepackage[margin=1in]{geometry} %
\usepackage[font=small,width=\linewidth]{caption}
\usepackage[hidelinks]{hyperref}
\usepackage[linesnumbered, algoruled, vlined, algo2e]{algorithm2e}
\SetKwInput{KwInput}{Input}
\SetKwInput{KwOutput}{Output}

\numberwithin{equation}{section}

\theoremstyle{plain}

\newtheorem{theorem}{Theorem}[section]
\newtheorem{lemma}[theorem]{Lemma}
\newtheorem{corollary}[theorem]{Corollary}
\newtheorem{remark}[theorem]{Remark}
\newtheorem{definition}[theorem]{Definition}
\newtheorem{proposition}[theorem]{Proposition}

\usepackage{booktabs}
\newcommand{\ra}[1]{\renewcommand{\arraystretch}{#1}}
\usepackage{mathabx}
\usepackage{enumitem}

\newcommand{\discretized}[1]{\ensuremath{\boldsymbol{#1}}}

\usepackage[
    backend=biber,
    natbib=true,
    style=numeric-comp,
    sorting=none,
    maxnames=2,
    giveninits=true,
    url=false,
    arxiv=pdf,
    doi=false,
    eprint=true,
    isbn=false,
    date=terse,
    block=ragged,
    maxbibnames=99
]{biblatex}
\providecommand{\keywords}[1]
{
  \small	
  \textbf{\textit{Keywords:}} #1
}

\title{Efficient approximation of high-dimensional exponentials by tensor networks}
\author{Martin Eigel, Nando Farchmin, Sebastian Heidenreich\\ and Philipp Trunschke}
\date{\today}

\begin{document}

\maketitle

\abstract{%
In this work a general approach to compute a compressed representation of the exponential $\exp(h)$ of a high-dimensional function $h$ is presented.
Such exponential functions play an important role in several problems in Uncertainty Quantification, e.g.\ the approximation of log-normal random fields or the evaluation of Bayesian posterior measures.
Usually, these high-dimensional objects are numerically intractable and can only be accessed pointwise in sampling methods.
In contrast, the proposed method constructs a functional representation of the exponential by exploiting its nature as a solution of a partial differential equation.
The application of a Petrov-Galerkin scheme to this equation provides a tensor train representation of the solution for which we derive an efficient and reliable a posteriori error estimator.
This estimator can be used in conjunction with any approximation method and the differential equation may be adapted such that the error estimates are equivalent to a problem-related norm.
Numerical experiments with log-normal random fields and Bayesian likelihoods illustrate the performance of the approach in comparison to other recent low-rank representations for the respective applications.
Although the present work considers only a specific differential equation, the presented method can be applied in a more general setting.
We show that the proposed method can be used to compute compressed representations of $\varphi(h)$ for any holonomic function $\varphi$.
}

\medskip
\noindent

\section{Introduction}%
\label{exptt:sec:introduction}

In this work we consider the problem of approximating the exponential $u=\exp(h)$ of a differentiable multivariate function $h(y)$ in the tensor train format, when the exponent $h$ is already given in this format.
Before presenting the new approach, we first illustrate the relevance of this often challenging task by two examples from the field of Uncertainty Quantification (UQ).

First, consider the steady state diffusion in a porous medium also known as Darcy's problem.
This is a standard benchmark problem in forward UQ and can be modeled by the second order linear partial differential equation
\begin{equation} \label{eq:darcy}
    -\operatorname{div}(\kappa \nabla w) = f,
\end{equation}
where $w$ is the concentration of some substance, the diffusion coefficient $\kappa$ determines the mobility of the particles and $f$ describes sources or sinks.
In practical applications the diffusion coefficient often takes the form
$\kappa(x,y) = \exp(\gamma(x,y))$ where the function $\gamma$ depends on the spatial coordinate $x$ as well as a random parameter $y$ that models uncertainties.
A popular approach for solving equation~\eqref{eq:darcy} is the \emph{stochastic Galerkin} (SG) method (cf.~\cite{LeMaitre_Knio,Ghanem_Spanos}) as it can be interpreted as an extension of the \emph{finite element method} (FEM) to parametric random PDEs.
In comparison to simpler sampling estimates such as different variants of the Monte Carlo method, the convergence of the SG method is potentially much faster since it exploits inherent structural properties of the considered problem, in particular anisotropic sparsity that can be captured in an appropriate generalized polynomial chaos basis.
Moreover, when using SG methods a reliable a posteriori error estimator can be computed, leading to a quasi optimal iterative construction of the discrete space as introduced in~\cite{EGSZ14,EGSZ15}.

However, to apply the SG method to the variational formulation of equation~\eqref{eq:darcy}, a functional representation of the diffusion coefficient $\kappa$ is required, which is provided by our proposed approach.

A second related example is the representation of some posterior probability density $\pi_\delta$ in the context of Bayesian inference~\cite{Stu10,kaipio2006statistical} via the exponential of the log-likelihood $\ell(y; \delta)$ and the prior $\pi_0$ as
\begin{align*}
     \frac{\mathrm{d}\pi_\delta}{\mathrm{d}\pi_0}(y)
     \,\propto\, \exp(\ell(y; \delta)) .
\end{align*}
In UQ this task arises for example in the parameter reconstruction of model data via inverse problems~\cite{Farchmin2020,mo2019deep,lassila2013reduced}.
If a functional representation of $\exp(\ell(y; \delta))$ can be constructed, it may for instance be used to efficiently generate independent posterior samples~\cite{dolgov2019TTdensities} or to compute high-dimensional quantities of interest such as moments or marginals~\cite{EGM20}.
Note that, if the covariance matrix for the sought distribution is too close to a singular matrix, it is reasonable to seek a suitable a basis transform before using a polynomial approximation of the probability density function (see e.g.~\cite{EGM20}).

\vspace{1ex}

The preceding two examples highlight the benefits of a functional representation of the occurring exponentials, which however is difficult to obtain.
A classical approach to represent any multivariate function $u : \mathbb{R}^M\to \mathbb{R}$ numerically is to choose a set of $d$ polynomial basis functions $\{p^{m}_i(y_m)\}_{i=0,\ldots,d-1}$ for each coordinate $y_m$ and each $m=1,\ldots M$, and then apply the expansion
\begin{equation*} \label{eq:intro:e_expansion}
    u(y) \approx \sum_{i_1=0}^{d-1}\cdots\sum_{i_M=0}^{d-1} c[i_1,\ldots,i_M]\, p^{1}_{i_1}(y_1)\cdots p^{M}_{i_M}(y_M).
\end{equation*}
Note that in both examples mentioned above, the dimension $M$ of the parameter vector $y$ may vary for practical problems from just a few to well over a hundred.
This renders a standard product basis representation like this unfeasible since $d^M$ coefficients would be required.
One remedy to this problem is to use a sparse representation of the coefficient tensor $c$.
This works well as long as the function $u$ can be sparsely represented in the chosen product basis.
However, to represent the exponential function $u(y) := \exp(h(y))$ higher-order multivariate polynomials are necessary and theoretical results indicate~\cite{Mugler2013} that the number of required coefficients in a sparse representation quickly becomes intractable in practice.
To mitigate this problem tensor network representations~\cite{Grasedyck2013} have been successfully applied to both examples.
This work focuses in particular on the \emph{tensor train} (TT) format~\cite{OT09}, also know as matrix product states, as a technique to compress the high-dimensional coefficient tensor $c$.
Given such a representation of the diffusion coefficient $\kappa$ it is easy to apply the SG algorithm~\cite{EMPS20}. If the likelihood function $\exp(\ell(y; \delta))$ is given in TT format quantities of interest~\cite{EGM20} can be computed promptly and samples can be drawn efficiently~\cite{dolgov2019TTdensities}.

To make the proposed approximation applicable to a wide range of problems, it is favourable to rely on non-intrusive techniques, which only require knowledge of the target function in some parameter realizations.
In contrast to e.g.\ the stochastic Galerkin Finite Element approach, these techniques have the advantage that they can utilize already existing code without the need to change any of its functionality.
Most renowned in this context are \emph{TT-Cross} methods~\cite{OT10,DS20,DS19} as well as the \emph{variational Monte Carlo} (VMC) method~\cite{ESTW19}.
Cross approximation based methods represent the function in a tensorized Lagrange interpolation basis and thus only require function evaluations in the interpolation points to obtain entries of the coefficient tensor.
The algorithm aims to provide a low-rank TT approximation of the coefficient tensor by means of a generalized skeleton decomposition.
If the low-rank assumption is satisfied, the target function has to be evaluated typically only on a small subset of the interpolation points, which is chosen actively by the cross approximation algorithm.
However, to the knowledge of the authors, there exists no bound for the required number of function evaluations to achieve a given error.
The VMC method on the other hand is a tensor regression technique that constructs an optimal low-rank approximation based on (random) training pairs of parameter realizations and target function evaluations.
The error of this method can be bounded by a constant times the best approximation error with high probability~\cite{EST21,T21}.
For a most general approach and a large degree of non-intrusiveness, we chose the VMC method since its application solely relies on samples generated a priori and does not require evaluations of the target function during runtime.
This allows for an easier applicability as in particular no interfaces to existing legacy (simulation) codes have to be provided.

Concerning the diffusion coefficient $\kappa$ in equation~\eqref{eq:intro:e_expansion}, there exist several approaches to represent the coefficient tensor $c$ in the TT format.
In~\cite{EMPS20} an exact representation is developed for the case that $\gamma(x,y)$ is an affine function in the parameters $y$.
However, since the construction is tailored to the specific structure of $\gamma$, it is not straight-forward to apply in other settings.
As an alternative, it is possible to construct an approximation in TT format by exploiting the connection of mean an variance of $\kappa$ and its exponent $\gamma$~\cite{EHLMW14} while also providing bounds for the approximation error in the Frobenius norm.
The use of numerical quadrature in~\cite{EHLMW14} is replaced by a block cross approximation algorithm, which constructs the TT format for $\kappa$ from a few evaluations of the entry-wise formula in~\cite{DKLM15} to improve efficiency.

To circumvent the dimension independent but slow convergence rate of Monte Carlo sampling in Bayesian inference, measure transport approaches~\cite{Marzouk2016,normalizingFlows,andrle2021invertible} gained a lot of popularity in recent years.
These approaches allow for the fast generation of independent samples from the posterior but utilizing the functional representation to compute quantities of interest such as mean values or higher-order moments analytically is difficult.
An additional drawback stems from the slow convergence of the (stochastic) gradient descent which usually is required to learn a measure transport.
Recent results to gain such a functional representation in TT format are presented in~\cite{EGM20,dolgov2019TTdensities}, which is motivated by the low-rank representability of Gaussian densities~\cite{rohrbach2020rank}.
However, the approximation task itself can be quite challenging, since one needs to guarantee small Lipschitz constants of the transport map~\cite{hagemann2020} to achieve numerical stability.

The above mentioned methods for the representation of an exponential are either highly intrusive or are at least tailored to a specific problem structure.
Furthermore, a direct approach of a solely sample based non-intrusive reconstruction using e.g.\ tensor recovery~\cite{ESTW19} or cross approximation~\cite{ballani2014tree,OT10} requires an infeasible amount of training data or an unrealistically good initial guess to obtain similar approximation results as the other methods.
One possibility to mitigate the requirement of a large number of samples is to utilize the scaling trick, i.e.\ to approximate $\exp(bh)$ for some $b\ll1$ and to compute $\exp(h) = \exp(bh)^{1/b}$.
Using the concept of functional tensor trains~\cite{Gorodetsky2019} could be a feasible approach to circumvent these issues as well as they do not rely on an a priori choice of a possibly ill-suited polynomial basis.
Unfortunately, to the best of our knowledge, no results for the problem at hand are known so far.

In this paper we propose to compute the exponential of a function given in TT format by solving the system
\begin{equation*}
        \nabla u - u\nabla h = f, \quad u(y_0) = 0
\end{equation*}
of \emph{partial differential equations} (PDEs) via Galerkin projection.
This makes it possible to derive an approximation of the exponential and establish error bounds via an equivalence of the discrete residuum and an energy norm.
The system operator and right-hand side can be discretized efficiently in the TT format which enables to solve the problem using the ALS method.
We restrict our examinations mainly to first order systems and to the representation of the exponential in the TT format.
However, a specific choice of the tensor network or the ODE is not necessary and both may be adapted to other applications.
In fact, we argue that our approach is not limited to the reconstruction of exponentials, but can be extended to other functions which constitute the solution of a linear homogeneous differential equation with polynomial coefficients, namely holonomic functions.
These functions are a powerful tool in computer algebra since they are smooth, can be represented by a finite amount of data and satisfy several closure properties.
By our work we extend these properties by adding an algorithmic representability in the TT format and hope that this will provide a new and practical way to represent functions in the ever more important field of high-dimensional numerical computations.

The rest of the paper is structured as follows.
After introducing basic notation in Section~\ref{sec:notation}, Section~\ref{sec:univariate} presents the general framework of our approach.
We describe the system of differential equations that we solve and the realization of the Galerkin scheme for the univariate case.
Subsequently, we define the energy norm and show equivalence of this norm to the discrete residual.
We conclude the section by generalizing the univariate results to the multivariate case and establish a theoretical foundation to apply our approach to other types of holonomic functions.
Section~\ref{sec:multivariate} recalls the basic workings of the ALS method and constructs a low-rank representation of the system operator and right-hand side in the TT format.
This representation scales linearly in the problem dimension and has bounded ranks dependent only on the ranks of the exponent $h$.
Additionally, we comment on the algorithmic realization of our method and give detail on some numerical intricacies in Section~\ref{sec:multivariate:algorithm}.
Numerical applications are discussed in Section~\ref{exptt:sec:experiments}.
There we explain our choice of discretization and error evaluation and give further details on the Darcy problem and Bayesian likelihoods before we discuss our numerical results.
The paper concludes with an outlook on further applications of the suggested approach as well as possible future research.

\subsection{Preliminaries}%
\label{sec:notation}

In the following we introduce the notation used throughout the paper.
If not specified differently, sets are denoted by calligraphic letters.
Denote by $[n]\subset\mathbb{N}_0:=\mathbb{N}\cup\{0\}$ the set of integers $\{0,\ldots,n-1\}$.
For sets $\mathcal{X}, \mathcal{Y}$, we denote by $C^{k,\alpha}(\mathcal{X};\mathcal{Y})$ the space of $k$-times differentiable functions from $\mathcal{X}$ to $\mathcal{Y}$ with $\alpha$-H\"older-continuous $k$\textsuperscript{th} derivative.
Denote by $L^p(\mathcal{X},\rho;\mathcal{Y})$ the weighted $L^p$-space for any $p\in\mathbb{N}$, weight function $\rho$, and two sets $\mathcal{X}$, $\mathcal{Y}$.
The same notation is used for Sobolev spaces $W^{k,p}(\mathcal{X},\rho;\mathcal{Y})$ and we write $H^k(\mathcal{X},\rho;\mathcal{Y}) = W^{k,2}(\mathcal{X},\rho;\mathcal{Y})$.
If $\rho\equiv 1$ we omit the weight function and if $\mathcal{Y}=\mathbb{R}$, we omit the image space.
The unit sphere in any normed space $\mathcal{X}$ is defined by
\begin{align*}
  S(\mathcal{X}) := \{ v\in\mathcal{X} \colon \Vert v \Vert_\mathcal{X} = 1 \}
\end{align*}
and the dual space of $\mathcal{X}$ is denoted by $\mathcal{X}^*$.
For any Hilbert space $\mathcal{X}$ and any subspace $\mathcal{Y}\subseteq\mathcal{X}$, the orthogonal projection onto $\mathcal{Y}$ is denoted by $P_{\mathcal{Y}}:\mathcal{X}\to\mathcal{Y}$.
We use standard notation for multiindices $\mu,\nu\in\mathbb{N}_0^M$ and additionally abbreviate sums over multiindices by
\begin{align*}
  \sum_{\nu=1}^{\mu}
  := \sum_{\nu_1=1}^{\mu_1} \dots \sum_{\nu_M=1}^{\mu_M}.
\end{align*}
Regular letters are used to notate standard non-discretized operators and functions.
The discretized versions of objects are denoted by boldface symbols.
For any fixed orthonormal set of functions
$\{ P_\mu \}_{\mu\in[d]^M} \subset L^2(\mathbb{R}^M,\rho)$
we define the finite dimensional subspace
\begin{align*}
  \mathcal{V}_d
  := \operatorname{span}\{P_\mu \colon \mu\in[d]^M\}
  \subset L^2(\mathbb{R}^M,\rho).
\end{align*}
Then, functions $w\in\mathcal{V}_d$ can be expressed by
\begin{align*}
  w(x) = \sum_{\mu\in[d]^M} \discretized{w}[\mu] P_\mu(x)
  \qquad\mbox{with}\qquad
  \discretized{w} \in \mathbb{R}^{d \times\dots\times d}.
\end{align*}
The space $\mathcal{V}_d$ is hence isomorphic to the space of coefficient tensors $\mathbb{R}^{d^M}$.

Note that the size of the coefficient tensor $\discretized{w}$ grows exponentially with the order $M$.
This is commonly referred to as the \emph{curse of dimensionality}.
To mitigate this exponential dependence on $M$, we employ a low-rank decomposition of the tensor $\discretized{w}$.
There are many tensor decompositions discussed in the
literature~\cite{Grasedyck2013,Hac12,Khoromskij2015,Kolda2009} but due to its simplicity and the wide availability in numerical libraries we have chosen the \emph{tensor train} (TT) format for our derivations.
It is one of the best-studied tensor formats in numerical mathematics (cf.~\cite{OT09,HRS12,rohrbach2020rank,EPS17}) and can efficiently represent all the tensors in our method.
Moreover, the employed optimization algorithm has been used reliably in many applications.
We want to stress, however, that the approach presented in this work can be applied to any other tensor decomposition.

In the following we provide a brief overview of the notation used with the tensor train format.
For further details, we refer the reader to~\cite{BSU16,EPS17} and the references therein.
The TT representation of a tensor $\discretized{w}\in\mathbb{R}^{d^M}$ is given as
\begin{align*}
  \discretized{w}[\mu]
  &= \sum_{k=1}^{r} \prod_{m=1}^{M} \discretized{w}_m[k_m,\mu_m,k_{m+1}]
  \qquad\mbox{for any }\mu\in[d]^M,
\end{align*}
with order three \emph{component tensors} (or \emph{cores}) $\discretized{w}_m\in\mathbb{R}^{r_m\times d\times r_{m+1}}$.
Here, $r=(r_1, \dots, r_{M+1})$, with $r_1=r_{M+1}=1$.
If all ranks $r_m$ are minimal, this is called \emph{tensor train decomposition} of $\discretized{w}$ with TT rank $r$.
The theoretical number of required degrees of freedom of a TT representation is given by
\begin{align}
  \label{eqn:multivariate:tt_rank}
  \operatorname{tt-dofs}(\discretized{w})
  = \sum_{m=1}^{M-1} \bigl(r_m dr_{m+1} - r_{m+1}^{2}\bigr) + r_M d,
\end{align}
which shows that the complexity of tensor trains behaves like $\mathcal{O}(Md\hat{r}^2)$ for $\hat{r}=\max\{ r_1,\ldots,r_M \}$.
In contrast to full tensor representations with complexity $\mathcal{O}(d^M)$, tensor trains depend only linearly on the order $M$.
As a result, the TT format is especially efficient for a small maximal rank $\hat{r}$.

In a similar fashion, for any $d,q\in\mathbb{N}$ we can express linear operators $W\colon \mathcal{V}_q\to\mathcal{V}_d$ in the tensor train format.
For this recall that the application of $W$ to $v\in\mathcal{V}_q$ reads
\begin{align*}
  Wv(x) = \sum_{\mu\in[d]^M} \sum_{\nu\in[q]^M} \discretized{W}[\mu,\nu]\discretized{v}[\nu] P_{\mu}(x).
\end{align*}
The TT representation of the tensor operator $\discretized{W}\colon\mathbb{R}^{q^M}\to\mathbb{R}^{d^M}$ is thus determined by
\begin{align*}
  \discretized{W}[\mu,\nu]
  &= \sum_{k=1}^{r} \prod_{m=1}^{M} \discretized{W}_m[k_m,\mu_m,\nu_m,k_{m+1}]
  \qquad\mbox{for any }\mu\in[d]^M \mbox{ and } \nu\in[q]^M,
\end{align*}
with the order four component tensors $\discretized{W}_m\in\mathbb{R}^{r_m\times d\times q\times r_{m+1}}$.
The TT decomposition always exists and can be computed using the hierarchical singular value decomposition (SVD)~\cite{Ose11}.
A truncated hierarchical SVD leads to quasi-optimal approximations of the TT decomposition in the Frobenius norm~\cite{OT09,Gra09,HS14}.
This can also be applied to tensors which are already represented in the TT format to obtain a TT decomposition with a lower rank.
This process is referred to as \emph{rounding}.
Note that sums and products can be computed efficiently in the TT format~\cite{Ose11}, which is crucial for the proposed method.
Moreover, many of the occurring tensors are of the form $\discretized{W} = \sum_{m=1}^M \discretized{B}_m$,
where, for any $\mu\in[d]^M$,
\begin{equation*}
  \discretized{B}_m[\mu]
  = \sum_{k=1}^{r} \prod_{n=1}^{m-1} \discretized{U}_n[k_n,\mu_n,k_{n+1}]
  \discretized{C}_m[k_m,\mu_m,k_{m+1}]
  \prod_{n=m+1}^M \discretized{V}_n[k_n,\mu_n,k_{n+1}] .
\end{equation*}
Such tensors are said to have a \emph{Laplace-like} structure~\cite{kazeev2012LaplaceLike} and are representable as tensors of rank $2r$.

To concisely present the Laplace-like structure later on, we define the following abbreviations.
Let $\discretized{W}_m$ and $\discretized{w}_m$ be component tensors of a TT operator and a TT tensor, respectively, and assume that the second dimension of both tensors is $d$.
Then $\discretized{W}_m^\intercal\discretized{w}_m$ denotes the contraction
\begin{equation}
    (\discretized{W}_m^\intercal\discretized{w}_m)[k_m\cdot\ell_m, \mu_m, k_{m+1}\cdot\ell_{m+1}]
    = \sum_{j=1}^d \discretized{W}_m[k_m,j,\mu_m,k_{m+1}] \discretized{w}_m[\ell_m,j,\ell_{m+1}],
\end{equation}
where $k\cdot\ell$ may be any bijection that maps a pair of rank indices $(k,\ell)$ onto a new rank index.
Similarily, if $\discretized{\tilde W}_m$ is a component tensor of another TT operator for which the second dimension is $d$, $\discretized{W}_m^\intercal\discretized{\tilde W}_m$ denotes the contraction
\begin{equation}
    (\discretized{W}_m^\intercal\discretized{\tilde{W}}_m)[k_m\cdot\ell_m, \mu_m, \nu_m, k_{m+1}\cdot\ell_{m+1}]
    = \sum_{j=1}^d \discretized{W}_m[k_m,j,\mu_m,k_{m+1}] \discretized{\tilde W}_m[\ell_m,j,\nu_m,\ell_{m+1}] .
\end{equation}
Note that the resulting tensor is of the same order as $\discretized{w}_m$ or $\discretized{\tilde W}_m$, i.e.\ $3$ or $4$, respectively.
Furthermore, we define the concatenation along the first dimension of two TT operator cores $\discretized{W}_m$ and $\discretized{\tilde W}_m$, if all but the first dimension have the same size and if the first dimension of $\discretized{W}_m$ has size $r_m$, by
\begin{align*}
    \begin{bmatrix}
        \discretized{W}_m \\ \discretized{\tilde W}_m
    \end{bmatrix}[k_m, \mu_m, \nu_m, k_{m+1}] = 
    \begin{cases}
        \discretized{W}_m[k_m,\hphantom{-r_m}\ \,\mu_m, \nu_m, k_{m+1}], &\mbox{if }k_m\leq r_m,\\
        \discretized{\tilde W}_m[k_m-r_m, \mu_m, \nu_m, k_{m+1}], &\mbox{if }k_m > r_m
    \end{cases}
\end{align*}
and similarily for two TT tensor cores $\discretized{w}_m$ and $\discretized{\tilde w}_m$.
If all but the last dimension have the same size and if the last dimension of $\discretized{W}_m$ has size $r_{m+1}$, we denote the concatenation along the last dimension by
\begin{align*}
    \begin{bmatrix}
        \discretized{W}_m & \discretized{\tilde W}_m
    \end{bmatrix}[k_m, \mu_m, \nu_m, k_{m+1}] = 
    \begin{cases}
        \discretized{W}_m[k_m, \mu_m, \nu_m, k_{m+1}\hphantom{-r_{m+1}}\ \,], &\mbox{if }k_{m+1}\leq r_{m+1},\\
        \discretized{\tilde W}_m[k_m, \mu_m, \nu_m, k_{m+1}-r_{m+1}], &\mbox{if }k_{m+1} > r_{m+1}.
    \end{cases}
\end{align*}
Combinations of the above are abbreviated in a straight-forward manner, e.g.\ by
\begin{align*}
    \begin{bmatrix}
        \begin{bmatrix}
            \discretized{W}_m & \discretized{\tilde W}_m
        \end{bmatrix}\\
        \begin{bmatrix}
            \discretized{\tilde{\tilde W}}_m & \discretized{\tilde{\tilde{\tilde{W}}}}_m
        \end{bmatrix}
    \end{bmatrix}
    =
    \begin{bmatrix}
        \discretized{W}_m & \discretized{\tilde W}_m \\
        \discretized{\tilde{\tilde W}}_m & \discretized{\tilde{\tilde{\tilde{W}}}}_m
    \end{bmatrix}
    \qquad\mbox{or}\qquad
    \begin{bmatrix}
        \begin{bmatrix}
            \discretized{W}_m \\
            \discretized{\tilde W}_m
        \end{bmatrix}\\
        \discretized{\tilde{\tilde W}}_m
    \end{bmatrix}
    =
    \begin{bmatrix}
        \discretized{W}_m \\
        \discretized{\tilde W}_m \\
        \discretized{\tilde{\tilde W}}_m
    \end{bmatrix}.
\end{align*}
Since $\discretized{W}_m, \discretized{\tilde{W}}_m, \discretized{\tilde{\tilde{W}}}_m$ and $\discretized{\tilde{\tilde{\tilde{W}}}}_m$ are tensors of order $4$ with equal second and third dimension ($d$ and $d'$, respectively), this notation can be interpreted as a standard concatenation of matrices over the ring $\mathbb{R}^{d\times d'}$.

\section{Approximation of exponentials via Galerkin projection}%
\label{sec:univariate}

In this section we demonstrate how the exponential of a function can be approximated by a Galerkin projection and derive computable bounds for the approximation error.
We emphasize again, that the presented orthogonal projection, as well as the derived a posteriori error bounds, can in principle be adapted to a broader class of multivariate holonomic functions.
These are functions that constitute solutions of a system of linear differential equations with polynomial coefficients as discussed in Section~\ref{sec:multivariate:operator}.

For any given exponent $h$, we construct a system of differential equations that has $\exp h$ as a unique solution and then use a Galerkin projection to construct an approximation to $\exp h$.
This also allows us to harvest well established results of the Galerkin method to obtain an a posteriori error control of the approximation~\cite{Bra07}.

We start with the description of the approach for univariate functions, subsequently derive upper and lower bounds of the approximation error in terms of the residual, and eventually generalize our results to the multivariate case.

\subsection{Approximation of univariate exponentials}%
\label{sec:univariate:setup}

Let $\rho$ be the standard Gaussian density and assume that the exponent in $C^1(\mathbb{R})\cap L^2(\mathbb{R},\rho)$ can be approximated by
\begin{align}
  \label{eqn:univariate:h_def}
  h(y) = \sum_{j=0}^{d_h-1} \discretized{h}[j] p_j(y),
\end{align}
where $\{p_j\}_{j=0}^{\infty}$ form an orthonormal basis in $L^2(\mathbb{R},\rho)$.
Consider the linear initial value problem
\begin{equation}
  \label{eqn:univariate:ode_inhomogeneous}
  \begin{aligned}
    u' - u\, h' &~= 0,\\
    u(y_0) &~= \exp h(y_0),
  \end{aligned}
\end{equation}
for an arbitrary $y_0\in\mathbb{R}$.
It is easy to verify that $u = \exp h$ is the unique solution to \eqref{eqn:univariate:ode_inhomogeneous}.
For $f(y) = \exp (h(y_0)) h'(y)$, the problem with inhomogeneous initial
condition~\eqref{eqn:univariate:ode_inhomogeneous} is equivalent to
\begin{equation}
  \label{eqn:univariate:ode}
  \begin{aligned}
    u' - u\, h' &~= f,\\
    u(y_0) &~= 0,
  \end{aligned}
\end{equation}
in the sense that $u$ is the solution of \eqref{eqn:univariate:ode} if and only
if $u + \exp h(y_0)$ is the solution of \eqref{eqn:univariate:ode_inhomogeneous}.
Although the choice of $y_0\in\mathbb{R}$ is arbitrary, it is advisable to choose the initial point such that $\rho(y_0)\gg0$ to avoid numerical precision issues.
Because of this and for the sake of simplicity, we assume $y_0=0$ in the following.

Let $\mathcal{X} = \{u \in H^2(\mathbb{R},\rho) \colon u(y_0) = 0\}$ and define the linear operator $B(v) = v' - v h'$.
The variational form of~\eqref{eqn:univariate:ode} then reads:
Find $u\in\mathcal{X}$ such that
\begin{align}
  \label{eqn:univariate:ode_var}
  ( B(u), v )_{\mathcal{V}} &~= (f, v )_{\mathcal{V}}
  \qquad\mbox{for all }v\in \mathcal{V},
\end{align}
where $(\bullet,\bullet)_{\mathcal{V}}$ denotes the inner product in $\mathcal{V} := L^2(\mathbb{R},\rho)$.
Since the weak solution $u$ of equation~\eqref{eqn:univariate:ode_var} is in $H^2(\mathbb{R}, \rho)$, it is regular enough to be a strong solution of the initial value problem~\eqref{eqn:univariate:ode}.
And since this problem satisfies the conditions of the Picard--Lindel\"of theorem, there exists a unique solution to the problem that depends continuously on the parameters.
Problem~\eqref{eqn:univariate:ode_var} is hence well posed.

For the Galerkin approximation of~\eqref{eqn:univariate:ode_var}, we define the ansatz space
$\mathcal{V}_{\mathrm{a}}:=\operatorname{span}\{p_j\colon j\in[d_a]\setminus\{0\}\}$ 
and the test space
$\mathcal{V}_{\mathrm{t}}:=\mathcal{V}_{d_t}$ for polynomial degrees $d_a,d_t\in\mathbb{N}_{>0}$.
With these spaces we can define discretized versions of $B$ and $f$ by
\begin{equation*}
    \discretized{B}_{ij} := (B(p_{j+1}), p_i)_{\mathcal{V}}
    \qquad\mbox{and}\qquad
    \discretized{f}_i := (f, p_i)_\mathcal{V} \qquad i\in[d_{\mathrm{t}}], \; j\in[d_{\mathrm{a}}] .
\end{equation*}
Consequently, the discretization of~\eqref{eqn:univariate:ode_var} reads
\begin{align*}
    \discretized{B} \discretized{u} = \discretized{f}.
\end{align*}

\begin{lemma} \label{lem:univariate:kerB}
    The operator $B\colon\mathcal{X}\to\mathcal{V}$ and the discretized operator $\discretized{B}\colon\mathbb{R}^{d_a}\to\mathbb{R}^{d_t}$ are injective for $d_{\mathrm{t}} > d_{\mathrm{a}}$, i.e.\ $\operatorname{ker}B = \{0\}$ and $\operatorname{ker}\discretized{B} = \{0\}$.
\end{lemma}

\begin{proof}
    The first assertion holds if and only if the solution $u=0$ of the ODE $B(u) = u' - h'u = 0$ is unique.
    Since $\mathcal{X}\subset C^{1,0}(\mathbb{R})$ we can consider this equation in the classical sense.
    The claim follows, because $h'$ is locally Lipschitz continuous and the ODE satisfies the conditions for the Picard--Lindelöf theorem. \\
    To prove the second assertion, first assume that $\operatorname{deg}(h) = 0$.
    Then $B u = u'$, which can only be zero when $u$ is constant.
    Since $p_0 \not\in \mathcal{V}_{\mathrm{a}}$ this can only be the case for $u=0$.
    Now assume $\operatorname{deg}(h) > 0$.
    Then $\operatorname{deg}(h'u) \ge \operatorname{deg}(u)$ and $\operatorname{deg}(u') = \operatorname{deg}(u)-1$.
    Consequently, $\operatorname{deg}(B(u)) \ge \operatorname{deg}(u)$ and $B(u) = 0$ implies $\operatorname{deg}(u) = 0$.
    The only polynomial of degree $0$ in $\mathcal{V}_{\mathrm{a}}$ is $0$.
    This concludes the proof.
\end{proof}

Lemma~\ref{lem:univariate:kerB} demonstrates the advantage of considering the problem with homogeneous initial conditions~\eqref{eqn:univariate:ode_var}.
The condition $u(y_0) = \exp h(y_0)$ effectively reduces the dimension of the solution space by $1$.
Solving with homogeneous initial conditions over the space $\mathcal{X}$ makes this explicit and ensures that the operator $B\colon\mathcal{X}\to L^2(\mathbb{R},\rho)$ is injective.
This is not the case when considering $B:H^1(\mathbb{R},\rho)\to L^2(\mathbb{R},\rho)$.

In the following we assume that $d_{\mathrm{t}} = d_{\mathrm{a}} + d_{\mathrm{h}}-1$.
This ensures that $B(u_{\mathrm{a}}) \in\mathcal{V}_{\mathrm{t}}$ for any $u_{\mathrm{a}}\in\mathcal{V}_{\mathrm{a}}$ and that $f\in\mathcal{V}_{\mathrm{t}}$.
Since the bound $\operatorname{deg}(B(u)) \le \operatorname{deg}(u)+\operatorname{deg}(h)-1$ is sharp, this is the smallest natural number with this property.
The resulting system $\discretized{B}\discretized{u} = \discretized{f}$ is overdetermined and can only be solved in a least-squares sense.
This can be obtained by performing a QR-factorization of the form
\begin{align*}
  \discretized{B} = 
  \begin{pmatrix} 
    \discretized{Q} & \discretized{Q}_{\bot} 
  \end{pmatrix}
  \begin{pmatrix} 
    \discretized{R} \\ 
    0 
  \end{pmatrix}
\end{align*}
and by solving the regular quadratic linear system $\discretized{Q}^\intercal\discretized{B}\discretized{u} = \discretized{Q}^\intercal\discretized{f}$.
From this it can be seen that solving $\discretized{Bu} = \discretized{f}$ is equivalent to a standard Galerkin method with the reduced test space $\tilde{\mathcal{V}}_{\mathrm{t}} = Q \mathcal{V}_{\mathrm{t}}$.
This test space is optimal in the sense that it minimizes the residual over $\tilde{\mathcal{V}}_{\mathrm{t}}^\bot$.

\subsection{A posteriori error bounds}%
\label{sec:univariate:error}

We are now interested in relating the discrete residual of the variational 
form~\eqref{eqn:univariate:ode_var} to the error in an appropriate norm.
The dynamical system~\eqref{eqn:univariate:ode_var} naturally induces a norm which we may use for this purpose.

We begin by considering the following lemma about injective linear operators.

\begin{lemma}\label{lem:univariate:energyNorm}
    Let $\mathcal{V}$ be a normed space and $W : \mathcal{X} \to \mathcal{V}$ be an injective linear operator.
    Then $\Vert v \Vert_W := \Vert W(v) \Vert_{\mathcal{V}}$ defines a norm on $\mathcal{X}$.
\end{lemma}

\begin{proof}
    Absolute homogeneity and the triangle inequality follow directly from the linearity of $W$ and from the fact that $\Vert \bullet\Vert_{\mathcal{V}}$ is a norm. 
    To show positive definiteness, let $v\in \mathcal{X}$ be such that $\Vert W(v) \Vert_{\mathcal{V}}=0$.
    With the injectivity of $W$, which yields $\operatorname{ker}W = \{0\}$, it directly follows
    that $v\equiv 0$.
\end{proof}

Since $B$ is injective by Lemma~\ref{lem:univariate:kerB}, it follows from Lemma~\ref{lem:univariate:energyNorm} that $\Vert\bullet\Vert_B$ is a norm.
In the broader context of variational methods for elliptic PDEs, we refer to this norm as the \emph{energy norm} induced by the system~\eqref{eqn:univariate:ode_var}.
With this we are subsequently able to prove that the discrete residual $\Vert \discretized{f}-\discretized{B}\discretized{u}_{\mathrm{a}}\Vert_2$ is equivalent to the error $\Vert u-u_{\mathrm{a}}\Vert_B$ in the energy norm for any discrete object $u_{\mathrm{a}}\in \mathcal{V}_{\mathrm{a}}$.
This is important for mainly two reasons.
First, by minimizing the discrete residual in the least-squares sense, we get a guarantee that the distance to the exact solution in the energy norm is minimized as well.
Second, for any discrete approximation $u_{\mathrm{a}}\in \mathcal{V}_{\mathrm{a}}$, the discrete residual represents a reliable and efficient estimator for the approximation error in the energy norm.
This, in principle, allows to adaptively control the number of steps of an iterative solver without any computational overhead. 

To show the equivalence, we first establish the relation of the discrete and continuous residual in the following lemma.

\begin{lemma}%
    \label{lem:univariate:dualNorm}
    Let $\mathcal{V}$ be a Hilbert space and $\mathcal{V}_{\mathrm{t}}\subseteq \mathcal{V}$.
    It holds for any $R\in\mathcal{V}^*$ that
    \begin{align*}
        \Vert R\Vert_{\mathcal{V}_{\mathrm{t}}^*}
        = \Vert P_{\mathcal{V}_{\mathrm{t}}} r\Vert_{\mathcal{V}},
    \end{align*}
    where $r\in\mathcal{V}$ denotes the Riesz representative of $R$ in $\mathcal{V}$ and $P_{\mathcal{V}_{\mathrm{t}}}$ is the orthogonal projection from $\mathcal V$ onto $\mathcal{V}_{\mathrm{t}}$.
\end{lemma}

\begin{proof}
    With the orthogonal projection $P_{\mathcal{V}_{\mathrm{t}}}$ and the dual pairing $\langle \bullet, \bullet\rangle_{\mathcal{V}_{\mathrm{t}}^*,\mathcal{V}_{\mathrm{t}}}$, we have
    \begin{align*}
        \Vert R\Vert_{\mathcal{V}_{\mathrm{t}}^*}
        = \sup_{v \in S(\mathcal{V}_{\mathrm{t}})}\vert\langle R, v\rangle_{\mathcal{V}_{\mathrm{t}}^*,\mathcal{V}_{\mathrm{t}}}\vert 
        = \sup_{v \in S(\mathcal{V}_{\mathrm{t}})}\vert\left(r, v\right)_{\mathcal{V}}\vert 
        = \sup_{v \in S(\mathcal{V}_{\mathrm{t}})}\vert\left(P_{\mathcal{V}_{\mathrm{t}}}r, v\right)_{\mathcal{V}}\vert.
    \end{align*}
    Since the supremum in the last equation is attained for
    $v=\Vert P_{\mathcal{V}_{\mathrm{t}}}r \Vert_{\mathcal{V}}^{-1} P_{\mathcal{V}_{\mathrm{t}}}r \in S(\mathcal{V}_{\mathrm{t}})$, the claim follows.
\end{proof}

As a consequence, it holds that $\Vert \mathcal{R}(v_{\mathrm{a}})\Vert_{\mathcal{V}_{\mathrm{t}}^*} = \Vert \discretized{f} - \discretized{B}\discretized{v}_{\mathrm{a}}\Vert_2$,
where for any $w\in\mathcal{X}$, $\mathcal{R}(w) := (f - B(w),\bullet)_{\mathcal{V}}\in\mathcal{V}^*$ is
the residual of~\eqref{eqn:univariate:ode_var}.
It remains to prove the equivalence of the continuous residual and the energy error, which is achieved with the following theorem.

\begin{theorem} \label{thrm:univariate:res_equiv_err}
    Let $u\in\mathcal{X}$ be the unique solution of~\eqref{eqn:univariate:ode_var},
    let $v_{\mathrm{a}}\in \mathcal{V}_{\mathrm{a}}$ be arbitrary and assume
    that $B(\mathcal{V}_{\mathrm{a}})\subseteq \mathcal{V}_{\mathrm{t}}$.
    Then it holds that
    \begin{align*}
        \Vert \mathcal{R}(v_{\mathrm{a}}) \Vert_{\mathcal{V}^*}
        = \Vert u-v_{\mathrm{a}}\Vert_B
        \quad\text{and}\quad
        \Vert \mathcal{R}(v_{\mathrm{a}})\Vert_{\mathcal{V}_{\mathrm{t}}^*} 
        \leq \Vert u-v_{\mathrm{a}}\Vert_B
        \leq \Vert \mathcal{R}(v_{\mathrm{a}})\Vert_{\mathcal{V}_{\mathrm{t}}^*} + \Vert P_{\mathcal{V}_{\mathrm{t}}^\perp} f \Vert_{\mathcal{V}} .
    \end{align*}
\end{theorem}

\begin{proof}
    The first assertion follows from the definition of the energy norm and the residual, i.e.
    \begin{align*}
        \Vert \mathcal{R}(v_{\mathrm{a}}) \Vert_{\mathcal{V}^*}
        = \Vert r \Vert_{\mathcal{V}}
        = \Vert B(u - v_{\mathrm{a}}) \Vert_{\mathcal{V}}
        = \Vert u-v_{\mathrm{a}}\Vert_B,
    \end{align*}
    where $r = f - B(v_{\mathrm{a}})$ is the Riesz representative of the residual $\mathcal{R}(v_{\mathrm{a}})$ in $\mathcal{V}$.
    To show the first inequality of the second assertion, note that Lemma~\ref{lem:univariate:dualNorm} directly yields
    \begin{equation*}
        \Vert \mathcal{R}(v_{\mathrm{a}})\Vert_{\mathcal{V}_{\mathrm{t}}^*}
        = \Vert P_{\mathcal{V}_{\mathrm{t}}}r\Vert_{\mathcal{V}}
        = \Vert P_{\mathcal{V}_{\mathrm{t}}}B(u - v_{\mathrm{a}})\Vert_{\mathcal{V}}
        \le \Vert B(u - v_{\mathrm{a}})\Vert_{\mathcal{V}}
        = \Vert u - v_{\mathrm{a}}\Vert_{B} .
    \end{equation*}
    The second inequality holds since $B(v_{\mathrm{a}}) \in\mathcal{V}_{\mathrm{t}}$, which implies
    $(r, P_{\mathcal{V}_{\mathrm{t}}^\perp} v)_{\mathcal{V}} = (f, P_{\mathcal{V}_{\mathrm{t}}^\perp} v)_{\mathcal{V}}$ and thus
    \begin{equation*}
        \Vert \mathcal{R}(v_{\mathrm{a}}) \Vert_{\mathcal{V}^*}
        \leq \sup_{v\in S(\mathcal{V})} \vert \left( r, P_{\mathcal{V}_{\mathrm{t}}}v\right)_{\mathcal{V}} \vert + \sup_{v\in S(\mathcal{V})} \vert ( r, P_{\mathcal{V}_{\mathrm{t}}^\perp} v)_{\mathcal{V}} \vert
        =  \Vert \mathcal{R}(v_{\mathrm{a}}) \Vert_{\mathcal{V}_{\mathrm{t}}^*} + \Vert  P_{\mathcal{V}_{\mathrm{t}}^\perp}f \Vert_{\mathcal{V}} .\qedhere
    \end{equation*}
\end{proof}

\begin{remark}%
    Theorem~\ref{thrm:univariate:res_equiv_err} is similar to well known results for  a posteriori error control in the context of elliptic PDEs~\cite{PR94,CDG14,Bra07} in the sense that estimating the residual in the dual norm of the discrete space $\mathcal{V}_{\mathrm{t}}$ introduces an additional data oscillation term.
    To guarantee the efficiency of the residual estimator, the right-hand side $f$ hence has to be resolved adequately.
    However, even if the data oscillation fails to be of higher order,
    the bound is always strictly efficient in the sense that
    \begin{align*}
        \Vert P_{\mathcal{V}_{\mathrm{t}}^\perp}f\Vert_{\mathcal{V}}
        = \Vert P_{\mathcal{V}_{\mathrm{t}}^\perp}B(u - v_{\mathrm{a}}) \Vert_{\mathcal{V}}
        \le \Vert u-v_{\mathrm{a}} \Vert_{B}. \qedhere
    \end{align*}
\end{remark}

Note that in our setting the right-hand side $f$ can be chosen rather freely.
Hence, without loss of generality we may assume that the data oscillation term can be neglected in applications.
Indeed, for the choice $f(y) = \exp h(y_0) h'(y)$ it holds that $f\in\mathcal{V}_{\mathrm{t}}$ and thus $\Vert P_{\mathcal{V}_{\mathrm{t}}^\perp}f\Vert_{\mathcal{V}}=0$.\\[1ex]

The following is an observation on how certain properties of the exponent $h$, such as regularity, influence the boundedness of the energy error with respect to other more meaningful or practical norms.

\begin{corollary}%
    \label{cor:univariate:boundedness_of_error}
    Let $u\in\mathcal{X}$ be the solution of~\eqref{eqn:univariate:ode_var}
    and let $h\in W^{1,\infty}(\mathbb{R})$.
    Then there exists $C>0$ such that
    \begin{align*}
        \Vert u-v_\mathrm{a} \Vert_B
        \leq C \Vert u-v_{\mathrm{a}} \Vert_{H^1(\mathbb{R},\rho)}
        \qquad\mbox{for all }v_a\in\mathcal{V}_{\mathrm{a}}.
    \end{align*}
    If $\rho$ is standard Gaussian and there exists $\varepsilon>0$, such
    that either $h'(y) \leq \frac{y}{2} + \varepsilon$ or
    $\frac{y}{2} + \varepsilon \leq h'(y)$ for all $y\in\mathbb{R}$,
    then for $c=\frac{1}{\varepsilon}$ it additionally holds that
    \begin{align*}
        \Vert u-v_a \Vert_{L^2(\mathbb{R},\rho)}
        \leq c\Vert u-v_\mathrm{a} \Vert_B
        \qquad\mbox{for all }v_a\in\mathcal{V}_{\mathrm{a}}.
    \end{align*}
\end{corollary}

\begin{proof}%
    The upper bound follows directly from the definition of the $H^1$ norm and
    the essential boundedness of $h'$ for
    $C=\sqrt{2}\max\{1, \Vert h'\Vert_{L^\infty(\mathbb{R})}\}$.
    To show the lower bound, let $\hat{c}(y) = \frac{y}{2}-h'(y)$.
    From the assumptions on $h'$ it follows that $\vert \hat{c}(y) \vert \geq \varepsilon$ for all $y\in\mathbb{R}$.
    Integrating by parts yields
    \begin{align*}
      \left( v', v \right)_{L^2(\mathbb{R},\rho)}
      = \left( \rho, v' v \right)
      = \frac{1}{2} \left( \rho, (v^2)' \right)
      = - \frac{1}{2} \left( \rho', v^2 \right)
      = \frac{1}{2} \left( y, v^2 \right)_{L^2(\mathbb{R},\rho)}.
    \end{align*}
    In combination with the boundedness of $\hat{c}$ this implies
    \begin{align*}
        \vert\left(B(v), v \right)_{L^2(\mathbb{R},\rho)}\vert
        = \vert\left(\hat{c}(y), v^2 \right)_{L^2(\mathbb{R},\rho)}\vert
        \geq \varepsilon \Vert v \Vert_{L^2(\mathbb{R},\rho)}^2.
    \end{align*}
    Since $\mathcal{V} = L^2(\mathbb{R},\rho)$, the residual can be bounded from below.
    For any $v\in\mathcal{X}$ it holds that
    \begin{align*}
      \Vert \mathcal{R}(v) \Vert_{\mathcal{V}^*}
      &= \sup_{w\in \mathcal{V}\setminus\{0\}} \frac{\vert \left( B(u-v), w \right)_{\mathcal{V}}\vert}{\Vert w\Vert_{\mathcal{V}}}
      \geq \frac{\vert \left( B(u-v), u-v \right)_{\mathcal{V}}\vert}{\Vert u-v\Vert_{\mathcal{V}}}\\
      &\geq \varepsilon \Vert u-v \Vert_{\mathcal{V}}.
    \end{align*}
    Theorem~\ref{thrm:univariate:res_equiv_err} then concludes the proof with
    $\Vert u-v_a \Vert_{L^2(\mathbb{R},\rho)} \leq \varepsilon^{-1}\Vert \mathcal{R}(v_{\mathrm{a}})\Vert_{\mathcal{V}^*} = \varepsilon^{-1}\Vert u-v_{\mathrm{a}}\Vert_B$
    for all $v_{\mathrm{a}}\in \mathcal{V}_{\mathrm{a}}$.
\end{proof}

\begin{remark}%
    The choice of the dynamical system for the approximation of $u$ is not
    unique and it determines the induced energy norm.
    A different choice of dynamical system may thus lead to more reasonable assumptions on the exponent $h$ then suggested by Corollary~\ref{cor:univariate:boundedness_of_error} to obtain bounds of the energy error by different norms.
    
    To illustrate this, consider for $x\in D=(0,1)$ the second order ODE
    \begin{align*}
        u'' &= (h'' + (h')^2)\, u 
        \qquad\mbox{in }D,\\
        u(y) &= \exp h(y)
        \qquad\qquad\mbox{on }\partial D.
    \end{align*}
    Homogenization and standard arguments for elliptic PDEs yield $\Vert u-v_a \Vert_B \approx \Vert u-v_a \Vert_{H_0^1}$ for all $v_a\in\mathcal{V}_{\mathrm{a}}\subset\mathcal{X}=H_0^1(D)$ if there exist $0 < \check{h} \leq \hat{h} < \infty$ such that $\check{h} \leq (h'' + (h')^2) \leq \hat{h}$.
    This is the case for many affine and quadratic exponents.
\end{remark}

\subsection{Generalization to multivariate exponentials}%
\label{sec:multivariate:operator}

In the multivariate setting we assume that for some $M,d_h\in\mathbb{N}$ the
exponent $h$ is given by an expansion
\begin{equation} \label{eqn:multivariate:h_full}
    h(y) = \sum_{\mu\in [d_{\mathrm{h}}]^M} \discretized{h}[\mu] P_{\mu}(y),
    \quad\text{where}\quad
    P_{\mu}(y) = \prod_{m=1}^M p_{\mu_m}(y_m) .
\end{equation}
Here, the orthonormal basis $\{P_\mu\}_{\mu\in\mathbb{N}_0^M}$ of the space $L^2(\mathbb{R}^{M},\varrho)$ for $\varrho(y)=\prod_{m=1}^M \rho(y_m)$ is chosen as the tensorization of the univariate orthonormal basis $\{p_j\}_{j\in\mathbb{N}_0}$ of $L^2(\mathbb{R};\rho)$.

The aim of this section is to generalize the univariate results from Section~\ref{sec:univariate:error} to the multivariate setting by considering a \emph{gradient system} of differential equations.

\begin{definition}[Gradient system]%
    \label{def:gradient_system}
    A system of first order linear differential equations of the form
    \begin{align*}
        \mbox{Find }u\in C^1(\mathbb{R}^M)\mbox{ such that}\quad
        \nabla u(y) = A(y) u(y) + F(y),
    \end{align*}
    with $A, F\in C(\mathbb{R}^M;\mathbb{R}^M)$ is called a \emph{gradient system} with $M$ \emph{component equations}
    \begin{align*}
        \partial_m u(y) = A_m(y) u(y) + F_m(y).
    \end{align*}
\end{definition}

For $M=2$, the simple example $A \equiv 0$ and $F(y_1, y_2) = (y_2\ \,0)^\intercal$ shows that an arbitrary gradient system may not have a solution.
However, the following theorem guarantees that the existence of a solution to a gradient system implies its uniqueness under suitable assumptions.

\begin{theorem}%
    \label{thrm:univariate:uniquenessGradientSystem}
    Consider the gradient system $\nabla u(y) = A(y) u(y) + F(y)$ on the convex domain $\Omega$ with initial condition $u(y^0) = u^0$ and assume that 
    $\sup_{y\in\tilde{\Omega}} \Vert A(y) \Vert < \infty$ for any closed subset $\tilde{\Omega}\subseteq\Omega$.
    Then, if there exists a classical solution $u$ to the gradient system, it is unique.
\end{theorem}

\begin{proof}
    To prove uniqueness, let $u_1$ and $u_2$ be two solutions to the gradient system. %
    For an arbitrary $y\in\mathbb{R}^M$ and $m=0,\ldots,M$ define $y^m[t] := (y_1,\ldots,y_{m-1},t,y^0_{m+1},\ldots,y^0_M)$ and $y^m := y^m[y_m]$ and note that $y^{m+1}[y^0_{m+1}] = y^m$ and $y^M = y$.
    Now consider the ordinary differential equation $\partial_{t}u(y^1[t]) = \partial_m u(y^1[t]) = A_1(y^1[t])u(y^1[t]) + F_1(y^1[t])$ with initial condition $u(y^1[y^0_1]) = u(y^0) = u^0$.
    The assumption that $\sup_{y\in\tilde{\Omega}} \Vert A(y) \Vert < \infty$ for any closed subset $\tilde{\Omega}\subseteq\Omega$ guarantees that the conditions of the Picard--Lindel\"of theorem are satisfied and the solution is (globally) unique.
    Since both $u_1$ and $u_2$ satisfy the equation, it follows that $u_1(y^1[t]) = u_2(y^1[t])$ for all $t\in\mathbb{R}$ and in particular $u^1 := u_1(y^1) = u_2(y^1)$.
    This argument can be iterated.
    By considering the equation $\partial_{t}u(y^{m+1}[t]) = A_{m+1}(y^{m+1}[t])u(y^{m+1}[t]) + F_{m+1}(y^{m+1}[t])$ with initial condition $u(y^{m+1}[y^0_m]) = u^m$
    it follows that $u^{m+1} := u_1(y^{m+1}) = u_2(y^{m+1})$.
    Finally, $u_1(y) = u_1(y^M) = u^M = u_2(y^M) = u_2(y)$.
    This implies $u_1\equiv u_2$, since $y\in\mathbb{R}^M$ was arbitrary.
\end{proof}

We now define the multivariate formulation of~\eqref{eqn:univariate:ode} as
\begin{equation} \label{eqn:multivariate:ode}
    \begin{aligned}
        \nabla u - u\nabla h &= f, \\
        u(y_0) &= 0,
    \end{aligned}
\end{equation}
where we choose $f(y) = \exp (h(y_{0})) \nabla h(y)$ and set $y_0=0\in\mathbb{R}^M$ as before.
Observe that $u(y) = \exp h(y)$ is a classical solution of~\eqref{eqn:multivariate:ode} and since the gradient system satisfies the conditions of Theorem~\ref{thrm:univariate:uniquenessGradientSystem}, this solution is unique.
For $k=\left\lceil\frac{M}{2}\right\rceil+1$, let $\mathcal{X} = \{u \in H^{k}(\mathbb{R}^M,\varrho) \colon u(y_0) = 0\}$ and $\mathcal{V} = L^2(\mathbb{R}^M, \varrho)^M$ and define the operator $B:\mathcal{X}\to\mathcal{V}$ by $B(v) = \nabla v - v \nabla h$.
Here, $\left\lceil x \right\rceil = \min\{n\in\mathbb{Z}\colon n \geq x\}$ denotes the ceiling function.
The variational form of this equation then reads: Find $u\in\mathcal{X}$ such that
\begin{align}
  \label{eqn:multivariate:ode_var}
  (B(u)_m, v )_{L^2(\mathbb{R}^M,\varrho)} &~= (f_m, v)_{L^2(\mathbb{R}^M,\varrho)}
  \qquad\text{for all }m\in[M]\text{ and }v\in L^2(\mathbb{R}^M,\varrho).
\end{align}
To formulate the Galerkin approximation of equation~\eqref{eqn:multivariate:ode_var}, we define the ansatz space $\mathcal{V}_{\mathrm{a}}$ by
\begin{equation*}
    \mathcal{V}_{\mathrm{a}}
    := \operatorname{span}\{P_\mu \,:\, \mu\in[d_{\mathrm{a}}]^M\setminus\{0\}\},
\end{equation*}
and set $\mathcal{V}_{d_t}$ as test space.
Note that $\mathcal{V}_{d_t}$ is the test space for the $m$\textsuperscript{th} component of the equation.
The test space for the complete operator $B$ is given by the Cartesian product
$\mathcal{V}_{\mathrm{t}} := \mathcal{V}_{d_t}^M$.
As above, we denote the discretized versions of $u$ by $\discretized{u}$ and define the discretization of $B$ and $f$ as
\begin{align*} %
    \discretized{B} = \begin{pmatrix} 
        \discretized{B}_1 &
        \dots &
        \discretized{B}_M
    \end{pmatrix}^\intercal
    \qquad\text{and}\qquad
    \discretized{f} = \begin{pmatrix} 
        \discretized{f}_1 &
        \dots &
        \discretized{f}_M
    \end{pmatrix}^\intercal
\end{align*}
for
\begin{align*} %
    \discretized{B}_{m}[{\mu,\nu}] = (B(P_\nu)_m, P_\mu)_{L^2(\mathbb{R}^M,\varrho)}
    \qquad\text{and}\qquad
    \discretized{f}_{m}[{\mu}] = (f_m, P_\mu)_{L^2(\mathbb{R}^M,\varrho)},
\end{align*}
where $\mu\in[d_{\mathrm{t}}]^M$ and $\nu\in[d_{\mathrm{a}}]^M\setminus\{0\}$.

\begin{remark}\label{rmk:multivariate:injectivity}
    By the Sobolev inequality, we have for any bounded open subset $U$ of $\mathbb{R}^M$ with $C^1$ boundary that $H^k(U)\hookrightarrow C^{1}(U)$ for $k\ge\left\lceil \frac{M}{2}\right\rceil+1$ and $Z_U := \inf_{y\in U} \varrho(y) > 0$.
    Hence, for any $f\in H^k(\mathbb{R}^M, \varrho)$,
    \begin{equation*}
        \Vert f\Vert_{H^k(U)}^2
        = \sum_{|\mu|\le k} \int_U |f^{(\mu)}(y)|^2 \,\mathrm{d}y 
        \le \sum_{|\mu|\le k} \int_U |f^{(\mu)}(y)|^2 Z_U^{-1} \varrho(y)\,\mathrm{d}y
        = Z_U^{-1} \Vert f\Vert_{H^k(U,\varrho)}^2 .
    \end{equation*}
    Thus, $f\in H^k(U)$ and consequently $f\in C^{1}(U)$
    for any $U$ in a countable covering of $\mathbb{R}^M$ by open sets with $C^1$ boundary.
    This shows that $f\in C^{1}(\mathbb{R}^M)$
    and hence every function in $\mathcal{X}$ is differentiable in the classical sense.
    This means that the weak solution of~\eqref{eqn:multivariate:ode_var} coincides with the classical solution and that $B$ is injective.
    Choosing $d_t \geq d_a + d_h -1$ ensures  $B(u_{\mathrm{a}})\in\mathcal{V}_{\mathrm{t}}$ for all $u_{\mathrm{a}}\in\mathcal{V}_{\mathrm{a}}$.
    Then $\discretized{B}$ is the matrix representation of the restriction of $B$ onto $\mathcal{V}_{\mathrm{a}}$ and thus injective.
\end{remark}

Remark~\ref{rmk:multivariate:injectivity} guarantees that the energy norm is well-defined and implies that $\Vert \mathcal{R}(v_{\mathrm{a}})\Vert_{\mathcal{V}_{\mathrm{t}}^*} = \Vert \discretized{f} - \discretized{B}\discretized{v}_{\mathrm{a}}\Vert_2$ where the residual is again defined by $\mathcal{R}(v_{\mathrm{a}}) := (f-B(v_{\mathrm{a}}), \bullet)_{\mathcal{V}}$.
To show equivalence of the residual to the energy norm and to obtain a posteriori error control note that Lemma~\ref{lem:univariate:dualNorm} and Theorem~\ref{thrm:univariate:res_equiv_err} also hold for the multivariate case.

To conclude this section we note that it is possible to generalize Theorem~\ref{thrm:univariate:uniquenessGradientSystem} to a larger set of what we refer to as multivariate holonomic functions, which is shown in the following proposition.

\begin{proposition} \label{thrm:multivariate:holonomic}
    Let $w$ be a holonomic function and $h$ be a polynomial.
    Then $w\circ h$ is the unique solution to a gradient system.
\end{proposition}

\begin{proof}
    Recall that a holonomic function $w$ of order $r$ is the solution to an $r$-th order homogeneous linear differential equation with polynomial coefficients.
    This means that there exist matrices $A(t),B(t)\in\mathbb{R}^{r\times r}$ such that
    \begin{align*}
        A(t)v'(t) + B(t)v(t) = 0
    \end{align*}
    and $w=v_1$.
    Let $\partial_m f(y)$ denote the partial derivative of the function $f$ with respect to $y_m$ in $y$ and define $V := v\circ h$.
    Then $\partial_m V(y) = \partial_m h(y)v'(h(y))$ for any $y\in\mathbb{R}^M$.
    This means that $V$ is the solution to the system of ordinary differential equations
    \begin{equation*}
        A(h(y))\partial_m V(y) + \partial_m h(y)B(h(y))V(y) = 0 .
    \end{equation*}
    To show uniqueness, let $\xi\in C^1(\mathbb{R};\mathbb{R}^M)$ and observe that the preceding system of equations implies
    \begin{equation*}
        A(h(\xi(t)))\partial_m V(\xi(t))\xi_m'(t) + \partial_m h(\xi(t))B(h(\xi(t)))V(\xi(t))\xi_m'(t) = 0
    \end{equation*}
    for all $m=1,\ldots,M$.
    Summing over $m$, the equation can be reformulated equivalently as
    \begin{equation*}
        A(h(\xi(t)))(V(\xi(t)))' + (h(\xi(t)))'B(h(\xi(t)))V(\xi(t)) = 0
    \end{equation*}
    or in shorter notation as $[A\circ h\circ \xi][V\circ\xi]' + [h\circ\xi]'[B\circ h\circ\xi][V\circ\xi] = 0$.
    This is a first order homogeneous linear differential equation for the function $V\circ\xi = v\circ [h\circ\xi]$.
    If $h$ is a polynomial and if $\xi(t) = \xi_y(t) := y_0 + (y-y_0)t$, this is a first order homogeneous linear differential equation with polynomial coefficients and exhibits a unique solution.
    Thus, if we are given an initial condition $V(y_0) = V_0$ and we set the initial condition $[V\circ\xi](0)=V_0$ then $V(y)=[V\circ\xi](1)$ is uniquely defined for any $y\in\mathbb{R}^M$.
\end{proof}

\section{Low-rank representation of the operator equation}%
\label{sec:multivariate}

This section is concerned with a low-rank discretization of the multivariate system $\discretized{B}\discretized{u}=\discretized{f}$ in order to make computations become feasible.
We first briefly illustrate how the \emph{alternating least squares algorithm} (ALS) can be employed to solve the high-dimensional linear system $\discretized{B}\discretized{u} = \discretized{f}$ in the TT format.
For simplicity, we utilize a plain ALS algorithm without rank adaptation, but other more sophisticated approaches, such as the \emph{density matrix renormalization group} (DMRG) method~\cite{Oseledets2012,oseledets_dmrg_2011} or some Riemannian optimization schemes~\cite{Steinlechner2016}, can be employed as well.

\subsection{The Alternating Linear Scheme}%
\label{sec:multivariate:ALS}

For the representation of a function $v_{\mathrm{a}}\in\mathcal{V}_{\mathrm{a}}$ in the tensor train format, we use the space $\mathcal{V}_{d_{\mathrm{a}}}$ and enforce the homogenous boundary condition $v_{\mathrm{a}}(y_0) = 0$ via
a regularization term.
Since the residual converges, the initial condition is satisfied up to a controllable error.
In the following we give a superficial recollection of the ALS algorithm and refer to~\cite{oseledets_dmrg_2011, holtz_alternating_2012} for a more detailed description.

Denote by $\discretized{P}\in \mathbb{R}^{d_a^M}$ the vector of basis functions evaluated in $y_0$, i.e.\ $\discretized{P}[\mu] := P_\mu(y_0)$ for any $\mu\in[d_a]^M$.
Since $\Vert \discretized{B}\discretized{u}-\discretized{f} \Vert_{2}^{2} = \sum_{m=1}^{M} \Vert \discretized{B}_m \discretized{u}-\discretized{f}_m \Vert_{2}^{2}$,
the regularized problem reads
\begin{align}
  \label{eqn:multivariate:global_system}
  \operatorname*{argmin}_{\discretized{u}\in\mathbb{R}^{d_a^M}}\quad \Vert \discretized{P}^\intercal\discretized{u} \Vert_{2}^{2} + \lambda \sum_{m=1}^{M} \Vert \discretized{B}_m\discretized{u}-\discretized{f}_m \Vert_{2}^{2},
\end{align}
where $\discretized{P}^\intercal\discretized{u}$ denotes the Frobenius-inner product of the two tensors $\discretized{P}$ and $\discretized{u}$.
The regularization parameter $\lambda$ controls the tradeoff between minimizing the residuum and enforcing the initial condition and has to be chosen by the practitioner.
For the sake of clarity and simplicity, however, we choose $\lambda=1$.
Inspired by the ALS, this functional can be minimized in an alternating fashion.
The tensor $\discretized{u}$ can be written in the tensor train format as $\discretized{u} := \discretized{Q}_k\discretized{U}_k$ where $\discretized{U}_k$ is the $k$-th component tensor and $\discretized{Q}_k$ is an operator that represents the contraction of this component tensor with the remaining tensor network.
The ALS then solves~\eqref{eqn:multivariate:global_system} by optimizing
\begin{align} \label{eqn:multivaraite:local_system}
  \operatorname*{argmin}_{\discretized{U}_k\in\mathbb{R}^{r_k\times d_a\times r_{k+1}}}\quad \Vert \discretized{P}^\intercal\discretized{Q}_k\discretized{U}_k \Vert_{2}^{2} + \sum_{m=1}^{M} \Vert \discretized{B}_m\discretized{Q}_k\discretized{U}_k-\discretized{f}_m \Vert_{2}^{2}
\end{align}
cyclically for each $k=1,\ldots,M$, until some convergence criterion is satisfied.
Each cycle is referred to as an ALS iteration step or \emph{sweep}.
The first order optimality condition of the optimization problem~\eqref{eqn:multivaraite:local_system} reads
\begin{equation*}
    \left(\discretized{Q}_k^\intercal \discretized{P}\discretized{P}^\intercal\discretized{Q}_k
    + \sum_{m=1}^M \discretized{Q}_k^\intercal \discretized{B}_m^\intercal\discretized{B}_m \discretized{Q}_k \right) \discretized{U}_k
    = \sum_{m=1}^M \discretized{Q}_k^\intercal \discretized{B}_m^\intercal\discretized{f}_m .
\end{equation*}
Note that this equation, as well as the subsequent formulas~\eqref{eqn:multivariate:ALS} and~\eqref{eqn:multivariate:ALSalgorithm}, have to be understood in a classical linear sense.
I.e.\ $\discretized{B}_m$, $\discretized{Q}_k$ and $\discretized{W}$ are classical linear operators and $\discretized{P}$, $\discretized{f}$ and $\discretized{b}$ are vectors.
Their nature as TT operators and TT tensors, respectively, is not represented.
The transpositions and contractions in these formulas should, in particular, not be confused with the operations on component tensors, defined in Section~\ref{sec:notation}.
Given the operator and right-hand side
\begin{align} \label{eqn:multivariate:ALS}
    \discretized{W} := \discretized{P}\discretized{P}^\intercal + \sum_{m=1}^{M} \discretized{B}_m^\intercal\discretized{B}_m
    \qquad\text{and}\qquad
    \discretized{b} := \sum_{m=1}^{M} \discretized{B}_m^\intercal\discretized{f}_m,
\end{align}
it is easy to find the minimum in~\eqref{eqn:multivaraite:local_system} by cyclically solving the linear system 
\begin{align}
    \label{eqn:multivariate:ALSalgorithm}
    \discretized{Q}_k^\intercal \discretized{W} \discretized{Q}_k \discretized{U}_k
    = \discretized{Q}_k^\intercal \discretized{b}
\end{align}
for each $k=1,\dots,M$.

\subsection{Low-rank representation of operator and right-hand side}%
\label{sec:multivariate:TTformat}

To construct an efficient representation of $\discretized{W}$ and $\discretized{b}$, we start by assembling $\discretized{B}_m$ and $\discretized{f}_m$.
For this we first define the partial derivative operator
\begin{align} \label{eqn:multivariate:D}
  \discretized{D}_m 
  = \discretized{I}^{\otimes (m-1)} \otimes \discretized{D} \otimes \discretized{I}^{\otimes (M-m)}
\end{align}
with the univariate differentiation operator $\discretized{D}[i,j] := (p_i, p_j')_{\mathcal{V}}$.
Now assume that the coefficient tensor $\discretized{h}$ of~\eqref{eqn:multivariate:h_full} can be represented in the TT format as
\begin{align*}
    \discretized{h}[\mu] = \sum_{k=1}^{r} \prod_{j=1}^{M} \discretized{h}_j[k_j,\mu_{j},k_{j+1}],
    \qquad\mu\in[d_h]^M,
\end{align*}
and define the multiplication operator
\begin{align} \label{eqn:multivariate:H_TT}
    \discretized{H}_m[\mu,\nu] 
    = \sum_{k=1}^{r}\prod_{j=1}^M \discretized{H}_{m,j}[k_j,\mu_j,\nu_j,k_{j+1}],
    \qquad \mu\in[d_\mathrm{t}]^M,\ \nu\in[d_\mathrm{a}]^M,
\end{align}
with the component tensors
\begin{align} \label{eqn:multivariate:H_TT_cores}
    \discretized{H}_{m,j}[k_j,\mu_j,\nu_j,k_{j+1}]
    &= \sum_{i_1=1}^{d_{\mathrm{h}}} \boldsymbol{\tau}[\mu_j,\nu_j,i_1] \discretized{h}_j[k_j,i_1,k_{j+1}]
    \qquad\text{for } j\ne m \text{ and}\\
    \label{eqn:multivariate:H_TT_cores:2}
    \discretized{H}_{m,m}[k_m,\mu_m,\nu_m,k_{m+1}]
    &= \sum_{i_1=1}^{d_{\mathrm{h}}} \sum_{i_2=1}^{d_{\mathrm{h}}}
    \boldsymbol{\tau}[\mu_m,\nu_m,i_1]\discretized{D}[i_1,i_2] \discretized{h}_m[k_m,i_2,k_{m+1}],
\end{align}
where $\discretized{\tau}[i,j,k] = (p_i p_j, p_k)_{\mathcal{V}}$ denotes the triple product tensor.
Since the partial derivative operators~\eqref{eqn:multivariate:D} are of rank $1$ and the multiplication operators~\eqref{eqn:multivariate:H_TT} are of rank $r$, the operators $\discretized{B}_m := \discretized{D}_m - \discretized{H}_m$ are of rank $(r+1)$ and $\discretized{f}_m := \discretized{D}_m\discretized{h}$ are of rank $r$.
For the initial condition, note that $\discretized{P}=\discretized{P}_1\otimes\dots\otimes\discretized{P}_M$ constitutes a rank one tensor with component tensors $\discretized{P}_j = (p_0(y_{0,j}),\dots,p_{d_a}(y_{0,j}))^\intercal$.

A naive computation of the sums in~\eqref{eqn:multivariate:ALS} would result in representation ranks that increase linearly in the number of parameters $M$.
However, we can exploit the structure of the operators $\discretized{W}$ and right-hand side $\discretized{b}$ which resembles the structure of Laplace-like operators~\cite{kazeev2012LaplaceLike}.
This allows us to bound the ranks of $\discretized{W}$ and $\discretized{b}$ independent of $M$.
Recall that $B_{m,j}$ denotes the $j$-th component tensor in the tensor train representation of $B_m$ and note that $B_{m_1,j} = B_{m_2,j}$ for any $j\in\mathbb{N}$ and $m_1,m_2\ne j$.
To emphasize this, we write $\discretized{C}_{j} := \discretized{B}_{m,j}$ when $m\ne j$ and $B_{j,j}$ otherwise.
Using the notation for the contraction and concatenation of component tensors, introduced in Section~\ref{sec:notation}, the component tensors of $\discretized{W}$ are then given by
\begin{align}
  \label{eqn:multivariate:W_TT_cores:W_1}
  \discretized{W}_1 &= 
  \begin{bmatrix}
    \discretized{P}_1^{\vphantom{\intercal}}
    \discretized{P}_1^\intercal & 
    \discretized{B}_{1,1}^\intercal 
    \discretized{B}_{1,1}^{\vphantom{\intercal}} & 
    \discretized{C}_1^\intercal 
    \discretized{C}_1^{\vphantom{\intercal}}
  \end{bmatrix},\\
  \label{eqn:multivariate:W_TT_cores:W_j}
  \discretized{W}_j &= 
  \begin{bmatrix}
    \discretized{P}_j^{\vphantom{\intercal}}
    \discretized{P}_j^\intercal & 0 & 0 \\ 0 &
    \discretized{C}_j^\intercal
    \discretized{C}_j^{\vphantom{\intercal}} & 0 \\ 0 &
    \discretized{B}_{j,j}^\intercal
    \discretized{B}_{j,j}^{\vphantom{\intercal}} &
    \discretized{C}_j^\intercal
    \discretized{C}_j^{\vphantom{\intercal}}
  \end{bmatrix}
  \qquad\text{for } j=2,\dots,M-1,\\
  \label{eqn:multivariate:W_TT_cores:W_M}
  \discretized{W}_M &= 
  \begin{bmatrix}
    \discretized{P}_M^{\vphantom{\intercal}}
    \discretized{P}_M^\intercal \\
    \discretized{C}_M^\intercal
    \discretized{C}_M^{\vphantom{\intercal}} \\
    \discretized{B}_{M,M}^\intercal
    \discretized{B}_{M,M}^{\vphantom{\intercal}}
  \end{bmatrix},
\end{align}
where $\discretized{P}_j\discretized{P}_j^\intercal$ are interpreted as tensors of order $4$ in $\mathbb{R}^{1\times d_{\mathrm{a}}\times d_{\mathrm{a}}\times 1}$.

From this it is easy to see that the rank of $\discretized{W}$ is given by $2(r+1)^2+1$ and is thus independent of $M$.
In an analogous way, the component tensors of $\discretized{b}$ are given by
\begin{align} \label{eqn:multivariate:b_TT_cores}
  \discretized{b}_1 = 
  \begin{bmatrix}
    \discretized{B}_{1,1}^\intercal
    \discretized{f}_{1,1}^{\vphantom{\intercal}} &
    \discretized{C}_1^\intercal
    \discretized{g}_1^{\vphantom{\intercal}}
  \end{bmatrix},\quad
  \discretized{b}_j = 
  \begin{bmatrix}
    \discretized{C}_j^\intercal
    \discretized{g}_j^{\vphantom{\intercal}} & 0 \\
    \discretized{B}_{j,j}^\intercal
    \discretized{f}_{j,j}^{\vphantom{\intercal}} & 
    \discretized{C}_j^\intercal
    \discretized{g}_j^{\vphantom{\intercal}}
  \end{bmatrix}
  \quad\mbox{and}\quad
  \discretized{b}_M = 
  \begin{bmatrix}
    \discretized{C}_M^\intercal
    \discretized{g}_M^{\vphantom{\intercal}} \\
    \discretized{B}_{M,M}^\intercal
    \discretized{f}_{M,M}^{\vphantom{\intercal}}
  \end{bmatrix},
\end{align}
for $j=2,\ldots,M-1$ and where $\discretized{g}_j := \discretized{f}_{m,j}$ for some $m\ne j$.
This shows that $\discretized{b}$ can be represented in the TT format with rank $2r(r+1)$ and is
independent of the dimension $M$ as well.

Finally, observe that the constant function $\exp h(y_0)$ can be represented by a TT tensor of rank one.
This means that the TT representation of the solution $u + \exp h(y_0)$ to the original problem can be computed in a straight-forward manner and the rank will increase by at most one.

\subsection{Algorithmic realization}%
\label{sec:multivariate:algorithm}

In the following we discuss some intricacies that arise in the application of an ALS to compute the Galerkin approximation~\eqref{eqn:multivariate:ALSalgorithm} in the TT format.
The method itself is rather straight-forward and we provide pseudo-code in Algorithm~\ref{alg:expTT}.
The method $\texttt{ALSsweep}(\discretized{W},\discretized{b},\discretized{u}_{\mathrm{a}})$ in line~\ref{alg:expTT:ALS} realizes one complete sweep of the ALS algorithm~\eqref{eqn:multivaraite:local_system}, i.e.\ it solves the local linear system~\eqref{eqn:multivariate:ALSalgorithm} for each component tensor $U_{1},\ldots,U_M$. %

\begin{algorithm2e}[htb]%
    \caption{Low-rank exponential approximation via Galerkin projection (\texttt{ExpTT})}%
    \label{alg:expTT}
    \KwInput{%
        TT representation of the exponent $h$,
        ansatz space dimension $d_a$,
        initial point $y_0$,
        stopping tolerance $\varepsilon$,
        and maximum number of iterations $N_\mathrm{ITER}$.
    }
    \KwOutput{%
        TT approximation $\discretized{u}_a$ of $\exp h$ and
        discrete relative residual $\mathrm{res}$.
    }
    Build operators $\discretized{D}_m$ and $\discretized{H}_m$ for $m=1,\dots,M$ according to~\eqref{eqn:multivariate:D} and~\eqref{eqn:multivariate:H_TT_cores}--\eqref{eqn:multivariate:H_TT_cores:2}.\\
    Use $\discretized{D}_m$, $\discretized{H}_m$ to assemble cores $\discretized{B}_{m,m}$, $\discretized{C}_m$, $\discretized{P}_m$, $\discretized{f}_{m,m}$ and $\discretized{g}_m$.\\
    Construct low-rank operator $\discretized{W}$ and right-hand side $\discretized{b}$ according to~\eqref{eqn:multivariate:W_TT_cores:W_1}--\eqref{eqn:multivariate:b_TT_cores}.\\
    Initialize the coefficient tensor $\discretized{u}_a$.\\ %
    \For{$j=1,\dots,N_\mathrm{ITER}$}{%
        Set $\discretized{u}_a = \texttt{ALSsweep}(\discretized{W}, \discretized{b}, \discretized{u}_a)$.\label{alg:expTT:ALS} \\
        \uIf{$\Vert \discretized{W}\discretized{u}_a-\discretized{b} \Vert_2\leq\varepsilon\Vert\discretized{b}\Vert_2$}{\textbf{break}.}
    }
    Set $\mathrm{res} = \Vert \discretized{B}\discretized{u}_a-\discretized{f} \Vert_2/\Vert \discretized{f}\Vert_2$.\label{alg:expTT:res} \\
    Build constant TT tensor $\discretized{c}=\exp h(y_0)$.\\
    Set $\discretized{u}_a = \discretized{u}_a + \discretized{c}$.\\
    \Return{$\discretized{u}_a$, $\mathrm{res}$.}
\end{algorithm2e}

Algorithm~\ref{alg:expTT} works in principle for any polynomial exponent $h$, but the resulting exponential might require large ansatz spaces and ranks.
This is a general problem of approximation methods and results in larger memory requirements and increase computational costs.

We propose to circumvent this problem by utilizing a simple scaling and squaring trick.
For a given scaling $s\in\mathbb{N}_{>0}$, we apply Algorithm~\ref{alg:expTT} to the scaled exponent $\tilde{h} := 2^{-s}h$ and compute the sought exponential via $\exp(h(y)) = \exp(\tilde{h}(y))^{2^s}$.
Since $\exp\circ\tilde{h}$ grows at a slower rate than $\exp\circ h$, this reduces the required ansatz space dimension.
A similar scaling approach is investigated in~\cite{Khoromskij2011} for Quantics Tensor Train approximations using a Taylor series expansion and Horner's rule to compute the scaled exponential instead of the Galerkin scheme proposed in this work.
Choosing the scaling $2^s$ minimizes the number of rescaling steps as we only need to compute the square $\exp(\tilde{h})^2$ $s$-times.
In doing so, we round and project $\exp(\tilde{h})^2$ onto $\mathcal{V}_{\mathrm{a}}$ in each step to prevent dimensions and ranks of the coefficient tensor from growing too much.
Other rescaling schemes are considered in~\cite{Khoromskij2011}, leading to better approximations at the cost of more rescaling steps.
As the rounding after each rescaling is the most time-consuming operation, we opt for the squaring approach to reduce the number of iterations.
Even though this trick introduces an additional error, we observe in our numerical examples that the memory issues can be overcome while still providing reasonably good results.
The resulting pseudo code is presented in Algorithm~\ref{alg:scale}.
Note that the output of Algorithm~\ref{alg:scale} exhibits the same ansatz space dimension as the one of Algorithm~\ref{alg:expTT}.
Nonetheless, this approach should be preferred because of two reasons.
First, the memory complexity of the TT tensor $\tilde{\discretized{u}}_a$ depends only linearly on the dimension $\tilde{d}_{\mathrm{a}}$ of the ansatz space while the complexity of the TT operator $\discretized{W}$ depends quadratically on $\tilde{d}_{\mathrm{a}}$.
This reduces the computational cost of applying Algorithm~\ref{alg:expTT} in line~\ref{alg:scale:expTT} of Algorithm~\ref{alg:scale}.
Second, multiplication can be performed efficiently in the TT format and it is easy to balance accuracy and computational cost.
This can be done by performing projections to lower dimensional discrete spaces or by rounding via a truncated SVD as in line~\ref{alg:scale:power} of Algorithm~\ref{alg:scale}.
\begin{remark}
    Scaling $h$ works well to reduce the ansatz space dimension and the rank of exponentials but may not work for other holonomic functions.
    However, we expect that similar tricks can be applied in these cases.
    For $\sin$ and $\cos$ for example a simple approach could be to reduce the frequency by approximating $u(h(sy))$ instead of $u(h(y))$ and to scale the basis functions afterwards.
\end{remark}

\begin{algorithm2e}[htb]%
    \caption{Scaled ExpTT}%
    \label{alg:scale}
    \KwInput{%
        TT representation of the exponent $h$,
        approximation dimensions $d_a$,
        scaling $s$,
        approximation dimensions for scaled problem $\tilde{d}_a$,
        initial condition $y_0$,
        stopping tolerance $\varepsilon$,
        rescaling tolerance $\varepsilon_s$,
        and maximum number of iterations $N_\mathrm{ITER}$.
    }
    \KwOutput{%
        TT approximation $\discretized{u}_a$ of $\exp h$.
    }
    Set $\tilde{h} = 2^{-s} h$.\\
    Compute $\tilde{\discretized{u}}_a, \mathrm{res} = \texttt{ExpTT}(\tilde{\discretized{h}}, \tilde{d}_a, y_0, \varepsilon, N_{\mathrm{ITER}})$.\label{alg:scale:expTT}\\
    Set $\discretized{u}_a = \tilde{\discretized{u}}_a$.\\
    \For{$j=1,\dots,s$}{
        Set $\discretized{u}_a = \discretized{u}_a \cdot \discretized{u}_a$.\label{alg:scale:power}\\
        Project $\discretized{u}_a$ onto $\mathcal{V}_{d_a}$ and round to tolerance $\varepsilon_s$.\\
    }
    \Return{$\discretized{u}_a$.}
\end{algorithm2e}

\section{Numerical experiments}%
\label{exptt:sec:experiments}

In this section we examine the numerical performance of the proposed Galerkin method to approximate the exponential of a function in TT format.
We compare the results of Algorithm~\ref{alg:scale} with those of a direct regression using the VMC approach~\cite{ESTW19}.
We also implemented the Taylor expansion approach described in~\cite{Khoromskij2011} but observed a similar performance as for the regression.
We hence refrain from a comparison to this approach in the following discussion.
To assess its practical potential, we investigate several benchmark problems common in Uncertainty Quantification.

First, we consider the reconstruction of a log-normal diffusion coefficient which appears frequently when modeling the porosity in the prototypical Darcy equation.
To test the performance of our algorithm, we setup a benchmark diffusion coefficient with small covariance length and compute the exponential of it's Karhunen-Lo\`eve expansion.

Second, we test our algorithm on a common benchmark diffusion coefficient used e.g.\ in~\cite{EGSZ14, dolgov2019TTdensities} to provide a comparison to state-of-the-art methods.
This problem exhibits several complications analytically~\cite{galvis-sarkis} and numerically~\cite{mugler2013convergence} and often is tackled with sampling techniques, in particular (multi-level) Monte Carlo methods~\cite{charrier2013finite,EMN}.
Functional (polynomial chaos) approaches were e.g.\ examined in~\cite{EMPS20} with an adaptive stochastic Galerkin FEM in TT format.
Stochastic collocation was e.g.\ used in~\cite{babuvska2010stochastic,nobile2016adaptive}.
In a third experiment, we investigate how our proposed method performs in approximating a density obtained by highly correlated and jointly Gaussian random variables, which is typically a difficult benchmark for tensorized approximation schemes but often encountered in real applications.
Finally, the recovery of the likelihood in the context of Bayesian inverse problems is considered, the theory of which can e.g.\ be found in~\cite{Stu10}.
Usually, again sampling methods are used for this often high-dimensional problem, the most popular of which certainly is the Markov chain Monte Carlo (MCMC) method.
Nevertheless, some recent developments showed that functional approximations of (posterior) densities are feasible and may prove beneficial in terms of convergence rates~\cite{EGM20, EMS18,dolgov2019TTdensities,rohrbach2020rank}.

These experiments are similar to those performed in~\cite{EMPS20,ESTW19,DS19,DKLM15,EHLMW14} for the log-normal diffusion coefficient and to those in~\cite{EGM20, EMS18, dolgov2019TTdensities} for the likelihood reconstruction.
We compare the approximation accuracy and computational time of our method to the results of other techniques from the literature\footnote{It has to be noted that the runtime of the previously reported experiments cannot be compared directly to what we observe with our (unoptimized) implementation. We nevertheless think that this provides a useful indication of the required computational effort.}, namely~\cite{EHLMW14,DKLM15,DS19}.

As weight function $\rho$ we choose the density of the standard Gaussian distribution.
As a basis for the trial and test spaces as well as for the parametrization of the exponent function $h$, we employ tensor products of normalized (probabilistic) Hermite polynomials.
One reason for this is that the triple product tensor $\kappa$ can be computed analytically~\cite{Ull08,Mal97,EMPS20}, which increases the overall computational performance.
Moreover, differentials of polynomials are explicitly known and thus cheap to compute and, in the case of Hermite polynomials, Hermite polynomials again.

\begin{proposition}%
  \label{lem:experiments:appell_sequence}
  For the normalized multivariate Hermite polynomials it holds that
  \begin{align*}
    \partial_{m} P_\nu = \sqrt{\nu_m}P_{\nu-e_m},
    \qquad\mbox{for all }m=1,\dots,M,
  \end{align*}
  where $e_m$ denotes the standard basis vector in $\mathbb{R}^M$.
  Moreover, the univariate differentiation operator $\discretized{D}$ is given analytically by $\discretized{D}_{ij} = \sqrt{j}\delta_{i,j-1}$.
\end{proposition}

\begin{proof}
  The normalized (probabilistic) Hermite polynomials of degree $k$ are given for any $j=1,\dots,M$ by $P^j_{k} = H_{k} / \sqrt{k!}$.
  Since the Hermite polynomials constitute an Appell sequence, i.e. $ H_k' = kH_{k-1}$ for all $k\in\mathbb{N}$, it holds $(P^j_{k})' = \sqrt{k}P^j_{k-1}$.
  Hence, ${(P_k^j, (P_\ell^j)')}_\rho = \sqrt{\ell}\delta_{k,\ell-1}$.
  The multiplicative structure of the tensorized Hermite polynomials completes the claim.
\end{proof}

Note that the product of two polynomials can always be represented by a polynomial of bounded degree. This implies that the triple product expansion is finite, i.e.\ that $\tau_{ijk}=0$ for any $k>i+j$.
This guarantees that we can choose ansatz and test spaces $\mathcal{V}_{\mathrm{a}}$ and $\mathcal{V}_{\mathrm{t}}$ such that the condition $B(\mathcal{V}_{\mathrm{t}}) \subseteq \mathcal{V}_{\mathrm{t}}$ of Theorem~\ref{thrm:univariate:res_equiv_err} is satisfied.

The finite element discretization is based on the open source package \texttt{FEniCS}~\cite{fenics} and all finite element computations use uniform triangulations of the unit square $D=[0,1]^2$.
The fully discretized problem $\discretized{W}\discretized{u}=\discretized{b}$ is solved in the TT format using Algorithm~\ref{alg:scale}, which relies on the TT representation and the ALS algorithm implemented in the open source tensor library \texttt{xerus}~\cite{xerus}.

Our approach and the VMC method~\cite{ESTW19} (a tensor regression technique employed several times in this section) rely on the choice of an initial guess.
Since this is chosen randomly, the obtained approximation as well as the CPU time for the computation of both methods may vary slightly for repeated runs of the same experiment.
Nevertheless, the deviations are minuscule and we hence refrain from a statistical assessment of the results.

\paragraph*{Computation of the error.}%

The relative discrete residual and (up to the data oscillation) the equivalent relative energy error is denoted by 
$\operatorname{res}(\discretized{u}_a) := \Vert \discretized{B}\discretized{u}_a-\discretized{f}\Vert_2\Vert \discretized{f}\Vert_2^{-1}$.
The residual is computed according to Algorithm~\ref{alg:expTT} based on the solution of~\eqref{eqn:multivariate:ode_var}.
If the exponent is scaled, i.e.\ if we apply Algorithm~\ref{alg:scale}, the relative residual of the scaled version of~\eqref{eqn:multivariate:ode_var} is considered.
For comparison, we additionally compute the absolute and relative $L^2$-errors via a Monte Carlo estimation.
For this, a set of $N_{\mathrm{MC}}$ independent samples $y^{(i)}\sim\mathcal{N}(0,I)$ is drawn.
The approximate solution $u_a\in\mathcal{V}_a$ obtained by our algorithm evaluated in the samples $y^{(i)}$ is compared to the corresponding (deterministic) sampled solution $u(y^{(i)})$.
The absolute and relative mean squared errors are approximated by a Monte Carlo quadrature for each $\hat v\in\mathcal{V}_a$,
\begin{equation}
    \label{eq:experiments:error_L2}
    \begin{aligned}
        \mathcal{E}_u(\hat v) &= \frac{1}{N_{\mathrm{MC}}} \sum_{i=1}^{N_{\mathrm{MC}}} \Vert u(y^{(i)}) - \hat v(y^{(i)})\Vert_*
        \quad\mbox{and}\quad \\
        \varepsilon_u(\hat v) &= \frac{1}{N_{\mathrm{MC}}} \sum_{i=1}^{N_{\mathrm{MC}}} \frac{\Vert u(y^{(i)}) - \hat v(y^{(i)})\Vert_*}{\Vert u(y^{(i)})\Vert_*},
    \end{aligned}
\end{equation}
where $\Vert \bullet \Vert_*$ is either the absolute value if $u(y)\in\mathbb{R}$ or $\Vert \bullet \Vert_{L^2(D)}$ if $u(y)\in H^1_0(D)$.
For the latter case, we additionally introduce the average relative $L^\infty$-error
\begin{align}
    \label{eq:experiments:error_Linf}
    \varepsilon_u^\infty(\hat v) = \frac{1}{N_{\mathrm{MC}}} \sum_{i=1}^{N_{\mathrm{MC}}} \frac{\Vert u(y^{(i)}) - \hat v(y^{(i)})\Vert_{L^\infty(D)}}{\Vert u(y^{(i)})\Vert_{L^\infty(D)}},
\end{align}
to allow a comparison to results of previous works.
The choice $N_{\mathrm{MC}}=10^{3}$ proved to be sufficient to obtain reliable estimates in our experiments.

\paragraph*{The random model problem.}%

The experiments we investigate concern the stationary random diffusion problem as described in~\cite{HOANG2014,GITTELSON2010,SG11} on the unit square $D=[0,1]^2$.
Concretely, for almost all $y\in\mathbb{R}^M$ we consider the random elliptic problem
\begin{equation}
    \label{exptt:eq:experiments:darcy}
    \begin{aligned}
        -\operatorname{div}( \kappa(x,y)\nabla w(x,y)) &= f(x), &\mbox{in }D,\\
        w(x,y) &= 0, &\mbox{on }\partial D.
    \end{aligned}
\end{equation}
For the sake of a clear presentation, the source term $f\in L^2(D)$ and the boundary conditions are assumed to be deterministic.
The diffusion coefficient $\kappa\colon D\times\mathbb{R}^M\to\mathbb{R}$ is typically considered log-normal and isotropic, i.e.\ $\log\kappa$ is an isotropic Gaussian random field~\cite{EHLMW14}.

Pointwise solvability of~\eqref{exptt:eq:experiments:darcy} for almost all $y\in\mathbb{R}^M$ is guaranteed by a Lax--Milgram argument in~\cite{galvis-sarkis,SG11}.
Well-posedness of the variational parametric problem is way more intricate and requires a larger solution space.
We refer to~\cite{SG11} for a detailed discussion.
Following the lines of e.g.~\cite{EGSZ14}, we assume a truncated Karhunen-Lo\`eve expansion of the affine exponent $\gamma=\log\kappa$ of the form
\begin{align}
    \label{eq:experiments:KLexpansion}
    \gamma(x,y)
    = \sum_{m=1}^{M} \gamma_m(x) y_m
    \qquad\mbox{for all }x\in D\mbox{ and almost all }y\in\mathbb{R}^M.
\end{align}
The expansion coefficient functions $\gamma_m$ enumerate all relevant planar Fourier modes in increasing total order and are given by
\begin{align}
    \label{eq:experiments:KLexpansion:h_m}
    \gamma_m(x)
    = \frac{9}{10\zeta(\sigma)} m^{-\sigma} \cos\bigl(2\pi\beta_1(m) x_1\bigr)\, \cos\bigl(2\pi\beta_2(m) x_2\bigr),
\end{align}
where $\zeta$ is the Riemann zeta function and for $k(m) = \lfloor -\frac{1}{2} + \sqrt{\frac{1}{4} +2m} \rfloor$,
\begin{align*}
    \beta_1(m) = m - k(m) \frac{k(m)+1}{2}
    \qquad\mbox{and}\qquad
    \beta_2(m) = k(m) - \beta_1(m).
\end{align*}
For our experiments we set a slow decay rate of $\sigma=2$.
For the deterministic discretization we choose either lowest order discontinuous Lagrange elements or continuous Lagrange elements.
However, other finite elements can be used with only slight adaptations as well.

\paragraph*{Bayesian log-likelihoods.}%

This section gives a short review of the Bayesian approach to inverse problems.
Its aim is to illustrate how our method can be used in this setting.
A comprehensive description on the Bayesian perspective on inverse problems can e.g.\ be found in~\cite{EGM20,Stu10,DS16}.

For an uncertain input $y\in\mathbb{R}^M$ consider the forward map
\begin{align*}
    \hat G\colon \mathbb{R}^M\to H_0^1(D),
    \qquad y\mapsto w(y), 
\end{align*}
where the model output $w(y)\in H_0^1(D)$ is chosen as the solution of~\eqref{exptt:eq:experiments:darcy}.
The inverse problem can then be formulated as
\begin{align}
    \label{eq:experiments:inverseProb}
    \mbox{For any given }\hat{w}\in H_0^1(D),\mbox{ find }y\in\mathbb{R}^M,
    \mbox{ such that }\hat G(y)=\hat{w}.
\end{align}
In practical applications, it is not possible to directly observe $\hat{w}\in H_0^1(D)$.
Hence, we assume that the measurement process of $\hat{w}$ is given by the bounded linear \emph{observation} operator
$\mathcal{O}\colon H_0^1(D)\to\mathbb{R}^J$ for some $J\in\mathbb{N}$.
These observations are usually either obtained directly from sensors or after a postprocessing step.
In our case, the observation operator describes the representation of a function in $H_0^1(D)$ by a finite element discretization with $J$ degrees of freedom.
When Courant FE are used, the degrees of freedom are equivalent to point observations in the domain related to the used mesh.
In most applications, exact (deterministic) solutions to~\eqref{eq:experiments:inverseProb} do not exist or are not unique, which implies that the inverse problem is
ill-posed.
A remedy is to introduce some kind of regularization to~\eqref{eq:experiments:inverseProb}.
The most commonly chosen probabilistic approach introduces a random additive centered Gaussian measurement noise $\eta\sim\mathcal{N}(0,\Sigma)$ with covariance $\Sigma\in\mathbb{R}^{J\times J}$.
With this, noisy observations are defined by
\begin{align}
    \label{eq:experiments:noisyForward}
    \delta = (\mathcal{O} \circ \hat G)(y) + \eta =: G(y) + \eta.
\end{align}
Under some mild assumptions on $G$, one can show a continuous version of the Bayes formula.
This yields the existence of a unique Radon--Nikodym derivative of the posterior measure $\pi_\delta$ of the conditional random variable $y\vert\delta$ with respect to the prior measure $\pi_0$ of $y$.
We refer to~\cite{Stu10,DS16} and~\cite{schwab2012sparse} for an analysis in the context of parametric PDEs.
Assuming the Gaussian noise $\eta$ is independent of $y$ this writes as
\begin{align}
    \label{eq:experiments:posterior}
    \frac{\mathrm{d}\pi_{\delta}}{\mathrm{d}\pi_0}(y) = Z^{-1} L(y; \delta)\, ,
    \qquad Z:=\mathbb{E}_{\pi_0}[L(y;\delta)],
\end{align}
with the likelihood $L(y;\delta):=\exp\ell(y;\delta)$, the negative Bayesian potential or log-likelihood
\begin{align}
    \label{eq:experiments:bayesPotential}
    \ell(y;\delta)
    := -\frac{1}{2} (\delta-G(y))\cdot\Sigma^{-1}(\delta-G(y)),
\end{align}
and a normalization constant $Z$, referred to as \emph{evidence}.

A surrogate for the forward map $G(y)$ can be computed in the TT fromat as presented in~\cite{EMS18,EPS17,EMPS20,ESTW19}.
From this it is easy to derive a representation of the log-likelihood $\ell$ in TT format by simple algebraic operations.
Our approach now provides the means to close the quite challenging, remaining gap to compute the TT representation of the likelihood $L$.
Given $L$ in the TT format, $Z$ can again be computed analytically.

\paragraph*{Initial condition for vector valued functions.}%

The construction of the operator $\discretized{B}$ and right-hand side $\discretized{f}$ in Section~\ref{sec:multivariate} is done with real-valued functions in mind.
However, this is not strictly required by our method.
As an example, for some $y\in\mathbb{R}^M$ consider a discretization of the exponent of the log-normal diffusion coefficient $\gamma(y)$ in a finite element space of dimension $J$ with the expansion
\begin{align*}
    \gamma(x,y) \approx \sum_{j=1}^J \sum_{\mu\in[d_h]^M} \discretized{\gamma}[j,\mu] \varphi_j(x) P_\mu(y)
    \qquad\mbox{for }x\in D\mbox{ and }y\in\mathbb{R}^M,
\end{align*}
where $\{\varphi_j\}_{j=1}^J$ is a basis of the finite element space.
The TT representation of $\discretized{\gamma}$ then reads
\begin{align*}
    \discretized{\gamma}[j,\mu]
    = \sum_{k=1}^r \discretized{\gamma}_0[j,k_1] \prod_{m=1}^M \discretized{\gamma}_m[k_m,\mu_m,k_{m+1}] .
\end{align*}
In this setting it is not straight-forward to perform the construction of the operator $\discretized{B}$ as described in~\eqref{eqn:multivariate:D}--\eqref{eqn:multivariate:H_TT}.
This is because the basis functions $\{\varphi_j\}_{j=1}^J$ for the deterministic mode depend on more than a single variable.
Hence, it may not be clear how to assemble the operator~\eqref{eqn:multivariate:D}.
Moreover, in case of a piecewise constant FE basis, the operator might not even be well-defined.
As an alternative, we choose a set $\{x^{(j)}\}_{j=1}^J$ of interpolation points for the FE space and build the operator $\discretized{B}$ and right-hand side $\discretized{f}$ pointwise for each finite element node $x^{(j)}$.
Here the interpolation points have to be chosen in such a way that a FE function can be recovered uniquely from its values at these points.
Since we use Lagrange FEM we use the Lagrange points together with the standard interpolation.
The resulting equations can be combined in a single system, which 
results in a slightly different operator $\discretized{W}$ and right-hand side $\discretized{b}$ but has no effect on the ranks.

\subsection{Lognormal field with given covariance length}%
\label{exptt:sec:experiments:darcy_results_KLE}

As a first benchmark we assume a diffusion field $\kappa(x,\omega)=\exp(\gamma(x,\omega))$ defined via the exponential of a Gaussian random field $\gamma(x,\omega)$ with covariance
\begin{align*}
  \operatorname{Cov}_{\gamma}(x,z) := c\, \exp\bigl(-\ell^{-2}\Vert x-z\Vert_2^2\bigr),
  \qquad x,z\in D,
\end{align*}
where $c>0$ is a scaling constant and $\ell>0$ is the isotropic covariance length.
Using the Karhunen-Lo\`eve expansion of the Gaussian field $\gamma$ yields the affine representation
\begin{align}
  \label{eq:experiments:KLE_cov_length}
  \gamma(x,\omega) = \sum_{j=1}^{\infty} \sqrt{\lambda_j}\phi_j(x) y_j(\omega),
  \quad\mbox{where}\quad
  \int_D \operatorname{Cov}_{\gamma}(x, z)\,\phi_m(z)\,\mathrm{d} z = \lambda_m \phi_m(x).
\end{align}
The random variables $y_j = y_j(\omega)$ are uncorrelated and jointly Gaussian.
More detail on this subject can be found in e.g.~\cite{Ghanem_Spanos}.
We truncate the affine expansion~\eqref{eq:experiments:KLE_cov_length} by choosing the largest $M\in\mathbb{N}$ eigenvalues such that $\lambda_1 \geq \lambda_{2} \geq \dots \geq \lambda_M$ and set $\gamma_m(x) = \sqrt{\lambda_m}\phi_m(x)$.
This leads to the truncated affine exponent $\gamma_M(x,y) = \sum_{m=1}^{M} \gamma_m(x) y_m$ for which we consider the exact TT representation with ranks $r_m=M+1$.
Additionally, we denote the exponential of the truncated affine field by $\kappa_M(x,y) = \exp(\gamma_M(x,y))$.
For this experiment we vary the isometric covariance length $\ell$ while setting the scale $c=10^{-2}$.
All simulations are conducted with a conforming FE space of piecewise affine Lagrange elements with $J=3017$ degrees of freedom on the L-shaped domain $[0,1]^2 \setminus (\frac{1}{2},1)^2$.
We choose $M=20$ parameters for the Karhunen-Lo\`eve expansion of $\gamma_M$ in our experiments to achieve a reasonable truncation error for all investigated $\ell$.
The relative $L^2$ error of the Karhunen-Lo\`eve approximation is computed with regard to an approximation using $\hat M=100$ expansion terms.

The reconstruction of the exponential is performed using Algorithm~\ref{alg:scale} with scaling $s=5$ and uniform polynomial degree $10$.
As initial points we choose the Lagrange interpolation points $x_0^{(1)},\dots,x_0^{(J)}$ and $y_0=0\in\mathbb{R}^M$ and set the rounding tolerance for the rescaling to $\varepsilon_s=10^{-10}$.
The ALS optimization stops if the residual reaches the threshold $\varepsilon=10^{-8}$.
To assess the results, we compare the reconstruction obtained by Algorithm~\ref{alg:scale} with a direct reconstruction using the \emph{variational Monte Carlo} (VMC) method~\cite{ESTW19}.
This method recovers the tensor train representation $\discretized{f}$ of a function $f$ from a given set of samples $\{(y_i, f(y_i)\}_{i=1}^{N_{\mathrm{VMC}}}$ by minimizing the least-squares loss
\begin{equation*}
    \underset{\discretized{f}}{\text{minimize}} \sum_{i=1}^{N_{\mathrm{VMC}}} \|f(y_i) - \discretized{P}(y_i)^\intercal \discretized{f} \| .
\end{equation*}
Here, $\discretized{P}(y_i)^\intercal \discretized{f}$ denotes the Frobenius inner product of the tensors $\discretized{P}(y_i)$ and $\discretized{f}$, and $\discretized{P}(y_i)$ is defined by $\discretized{P}(y_i)[\mu] := P_\mu(y_i)$.
We use $N_{\mathrm{VMC}}=10^4$ randomly generated training data pairs $(y^{(i)}, \kappa(y^{(i)}))$ in the experiments.

\begin{table}[htpb]%
    \centering
    \ra{1.1}  %
    \begin{tabular}{cccccccc}\toprule
      $\ell^2$
      & $\varepsilon_{\gamma_{\hat M}}(\gamma_M)$
      & $r_{\mathrm{max}}(\kappa_{\mathrm{a}})$
      & $\operatorname{res}(\kappa_\mathrm{a})$
      & $\varepsilon_{\kappa_M}(\kappa_{\mathrm{a}})$
      & $r_{\mathrm{max}}(\kappa_{\mathrm{VMC}})$
      & $\operatorname{res}(\kappa_\mathrm{VMC})$
      & $\varepsilon_{\kappa_M}(\kappa_{\mathrm{VMC}})$\\
      \midrule
      $10$  & $1.59\cdot 10^{-7}$ & $21$  & $1.62\cdot 10^{-4}$ & $4.10\cdot 10^{-7}$ & $9$  & $4.97\cdot 10^{-1}$ & $6.59\cdot 10^{-3}$ \\
      $5$   & $9.30\cdot 10^{-7}$ & $21$  & $1.99\cdot 10^{-4}$ & $1.68\cdot 10^{-6}$ & $7$  & $5.51\cdot 10^{-1}$ & $3.01\cdot 10^{-3}$ \\
      $3$   & $3.45\cdot 10^{-6}$ & $23$  & $1.30\cdot 10^{-4}$ & $3.78\cdot 10^{-6}$ & $8$  & $3.45\cdot 10^{-1}$ & $8.62\cdot 10^{-3}$ \\
      $1$   & $5.79\cdot 10^{-5}$ & $38$  & $1.31\cdot 10^{-4}$ & $8.04\cdot 10^{-6}$ & $8$  & $4.87\cdot 10^{-1}$ & $1.05\cdot 10^{-2}$ \\
      $0.8$ & $1.03\cdot 10^{-4}$ & $42$  & $1.65\cdot 10^{-4}$ & $1.02\cdot 10^{-5}$ & $7$  & $2.99\cdot 10^{-1}$ & $1.24\cdot 10^{-2}$ \\
      $0.5$ & $3.27\cdot 10^{-4}$ & $52$  & $2.10\cdot 10^{-4}$ & $2.69\cdot 10^{-5}$ & $7$  & $3.09\cdot 10^{-1}$ & $1.35\cdot 10^{-2}$ \\
      $0.3$ & $1.05\cdot 10^{-3}$ & $67$  & $2.62\cdot 10^{-4}$ & $1.99\cdot 10^{-4}$ & $5$  & $7.38\cdot 10^{-1}$ & $3.86\cdot 10^{-2}$ \\
      $0.1$ & $8.57\cdot 10^{-3}$ & $114$ & $5.28\cdot 10^{-4}$ & $3.21\cdot 10^{-5}$ & $13$ & $1.30\cdot 10^{0}$  & $6.17\cdot 10^{-2}$ \\
      \bottomrule
    \end{tabular}
    \caption{
      Maximum approximation rank $r_{\mathrm{max}}$, relative residual $\operatorname{res}$ and $L^2$-error $\varepsilon_{\kappa_M}$ for the reconstruction obtained via Algorithm~\ref{alg:scale} and directly by the VMC method in relation to a change in isotropic covariance length $\ell$.
      The computation is carried out on a uniform triangulation of the L-shaped domain with $5776$ triangles ($3017$ FE DoFs) for $M=20$ parameters with uniform polynomial degree $d_\mathrm{a}=10$.
      The relative $L^2$-error of $\gamma_M$ with respect to a more accurate approximation $\gamma_{\hat M}$ for $\hat M = 100$ is given for reference.
    }%
    \label{tab:experiments:logNormal:cov_length}
\end{table}

Table~\ref{tab:experiments:logNormal:cov_length} shows the relative residual and $L^2$-error~\eqref{eq:experiments:error_L2} for our method as well as the VMC reconstruction for different covariance lengths.
The numerical results indicate that the maximum ranks of $\kappa_\mathrm{a}$ grow extensively as the covariance length $\ell$ decreases whereas such strong growth cannot be observed for the ranks of $\kappa_{\mathrm{VMC}}$.
The computed residuals for both approximations are almost insensitive to the decrease in $\ell$, whereas the relative $L^2$-errors of both approximations increase slightly.
Additionally, we observe that the error of the approximation generated by Algorithm~\ref{alg:scale} is at least three orders of magnitude smaller than the direct VMC approximation concerning the global $L^2$-error $\varepsilon_{\kappa_M}$.
The relative residual is three orders of magnitude smaller for the output generated by Algorithm~\ref{alg:scale} as well, which indicates that the residual behaves roughly proportional to $\varepsilon_{\kappa_M}$.
All these observations are coherent with the expectations and reflect the difficulty of approximating the exponential of fields with small covariance length.
Furthermore, we notice that Algorithm~\ref{alg:scale} produces approximations with larger (maximum) ranks than the VMC method.
We take this as an indication that these larger ranks are required to adequately represent the exponential and assume that similar results can be obtained by a direct computation via the VMC method with a (drastically) larger number of training samples.
We also observe that the relative error $\varepsilon_{\gamma_{\hat M}}(\gamma_M)$ increases dramatically as $\ell$ decreases, which again is expected.
Algorithm~\ref{alg:scale} always yields approximations to $\kappa_M$ of at least the same order of magnitude as $\varepsilon_{\gamma_{\hat M}}(\gamma_M)$, whereas the approximation error $\varepsilon_{\kappa_M}(\kappa_{\mathrm{VMC}})$ dominates $\varepsilon_{\gamma_{\hat M}}(\gamma_M)$ by one to four orders of magnitude depending on $\ell$.
Reconstruction times to obtain an approximation by Algorithm~\ref{alg:scale} are less then $10$ minutes in all investigated cases.%

\subsection{Lognormal Darcy diffusion coefficient}%
\label{exptt:sec:experiments:darcy_results}

In this section we investigate the approximation of the log-normal diffusion coefficient $\kappa$ of~\eqref{exptt:eq:experiments:darcy}, which for instance can be used in a stochastic Galerkin scheme.
For the experiments conducted in this section we choose to
discretize the diffusion field $\kappa\in L^2(\mathbb{R}^M,\rho;L^\infty(D))$ with conforming first order Lagrange finite elements for varying degrees of freedom (DoF) and stochastic dimensions $M$.
We focus on these first order elements, since we observe that a comparison with order zero discontinuous and higher order continuous Lagrange elements yields similar results.
Hence, the choice of polynomial order for the spatial component seems to have no influence on the approximation quality of the exponential.
The exponent $\gamma=\log \kappa$ is approximated in the same finite element space as $\kappa$.
Since $\gamma$ is an affine function in the stochastic variables $y$ (cf.~\eqref{eq:experiments:KLexpansion}), we set $d_h=2$.
To obtain an approximation of the exponent $\gamma$ in TT format we again employ the VMC method.
Here, it is in principle possible to find an exact representation of the affine exponent $\gamma$.
Nevertheless, we choose an approximation via VMC for two reasons.
First, the non-intrusive character of VMC allows for easy adaptation to other more complicated problems, which is why we expect this to be commonly done, even if it is feasible to obtain exact representations in specific cases.
Second, there might not exist an exact TT representation for other applications or it might be very intricate to derive.
The choice of an inexact representation of $\gamma$ thus demonstrates the practical relevance of our method due to a broad applicability.
Additionally, since Theorem~\ref{thrm:univariate:res_equiv_err} holds for any approximation of $\kappa$, this is a good opportunity to confirm our theoretical results.

The VMC method only requires evaluations of $\gamma$ in realizations $\{ y^{(i)} \}_{i=1}^{N_{\mathrm{VMC}}}$ to find a low-rank approximation of a function in TT format.
We increase $N_{\mathrm{VMC}}$ as $M$ gets larger to obtain approximations $\gamma_{\mathrm{VMC}}$ of $\gamma$ with relative error $\varepsilon_\gamma(\gamma_{\mathrm{VMC}})\leq 10^{-8}$ for all $M$ depicted in Tables~\ref{tab:experiments:logNormal:M_inc} and~\ref{tab:experiments:logNormal:DoF_inc}.
The approximations of $\kappa$ are computed for uniform polynomial degree $d_a = 10$ for each stochastic component via Algorithm~\ref{alg:scale} with scaling number $s=5$.
We use the Lagrange interpolation points $x_0^{(1)},\dots,x_0^{(J)}$ and $y_0=0\in\mathbb{R}^M$ as initial points.
As the stopping tolerance for all experiments we set $\varepsilon=10^{-8}$ and round the rescaling of the approximation to $\varepsilon_s=10^{-7}$ in each iteration (cf.\ Algorithm~\ref{alg:scale}).

\begin{table}[htpb]%
    \centering
    \ra{1.1}  %
    \begin{tabular}{ccccccc}\toprule
      $M$ &
      $\varepsilon_\gamma(\gamma_{\mathrm{VMC}})$ &
      $\operatorname{res}(\kappa_\mathrm{VMC})$ &
      $\varepsilon_\kappa^\infty(\kappa_{\mathrm{VMC}})$ &
      $\operatorname{res}(\kappa_a)$ &
      $\varepsilon_\kappa^\infty(\kappa_a)$ &
      time [s] \\
      \midrule
      $5 $ & $3.33\cdot 10^{-9}$ & $1.33$ & $1.42\cdot 10^{-3}$ & $3.32\cdot 10^{-3}$ & $7.01\cdot 10^{-5}$ & $ 110.05$ \\
      $10$ & $5.82\cdot 10^{-9}$ & $4.35$ & $2.77\cdot 10^{-2}$ & $1.18\cdot 10^{-3}$ & $2.83\cdot 10^{-5}$ & $ 201.43$ \\
      $15$ & $1.41\cdot 10^{-9}$ & $2.78$ & $1.92\cdot 10^{-2}$ & $6.66\cdot 10^{-4}$ & $1.95\cdot 10^{-5}$ & $ 590.92$ \\
      $20$ & $2.68\cdot 10^{-9}$ & $1.88$ & $1.76\cdot 10^{-2}$ & $3.64\cdot 10^{-4}$ & $1.80\cdot 10^{-5}$ & $1865.92$ \\
      \bottomrule
    \end{tabular}
    \caption{Relative approximation errors and computation time for the
      approximation of the log-normal diffusion coefficient $\kappa$ for different
      numbers of stochastic parameters $M$. The computation is done on a uniform
      triangulation of $D$ with $5000$ triangles ($2601$ FE DoFs) and uses
      stochastic polynomials of degree $10$ or less for each mode. Here, $\kappa_{\mathrm{VMC}}$
      is an approximation of $\kappa$ obtained via direct VMC and
      $\kappa_a$ is the output of Algorithm~\ref{alg:scale}.}
    \label{tab:experiments:logNormal:M_inc}
\end{table}

Table~\ref{tab:experiments:logNormal:M_inc} shows errors of the approximations $\gamma_{\mathrm{VMC}}$ and $\kappa_{\mathrm{VMC}}$ obtained via the VMC method and of the output $\kappa_a$ of Algorithm~\ref{alg:scale:expTT} for different expansion lengths $M$.
Algorithm~\ref{alg:expTT} converges in less than $10$ iterations to the prescribed tolerance of $\varepsilon=10^{-8}$.

When using the generic VMC approach to directly reconstruct an approximation $\kappa_{\mathrm{VMC}}$ of $\kappa$ from samples with the same stochastic dimensions and $N_{\mathrm{VMC}}=10^{4}$, the relative error of $\varepsilon_\kappa^\infty(\kappa_{\mathrm{VMC}})$ seems to stagnate independent of $M$ at about $10^{-2}$, which exceeds the error of our method by three orders of magnitude.

Even though the exponent $\gamma$ does not satisfy the conditions of Corollary~\ref{cor:univariate:boundedness_of_error}, the relative discrete residual $\operatorname{res}(\bullet)$ behaves (up to a multiplicative constant) similarly to $\varepsilon_{\kappa}^\infty(\bullet)$ independent of the number of modes $M$, the degrees of freedom of the FE space or the reconstruction method.

The error $\varepsilon_{\kappa}^\infty(\kappa_a)$ is comparable to the approximation results of~\cite{DS19} and about one order of magnitude smaller then the ones reported in~\cite{EHLMW14,DKLM15}, which suggests that our method compares favourably to these state of the art algorithms.
Table~\ref{tab:experiments:logNormal:DoF_inc} shows errors and computation times for the reconstruction of $\kappa$ for a fixed number of modes $M=20$ and an increasing number of FE degrees of freedom.
The dimension of the finite element space does not seem to have any influence on either the relative approximation error $\varepsilon_\kappa^\infty(\kappa_a)$ or the discrete residual $\operatorname{res}(\kappa_a)$.

\begin{table}[htpb]%
    \centering
    \ra{1.1}  %
    \begin{tabular}{ccccccc}\toprule
      $\operatorname{DoFs}$ &
      $\varepsilon_\gamma(\gamma_{\mathrm{VMC}})$ &
      $\operatorname{res}(\kappa_\mathrm{VMC})$ &
      $\varepsilon_\kappa^\infty(\kappa_{\mathrm{VMC}})$ &
      $\operatorname{res}(\kappa_a)$ &
      $\varepsilon_\kappa^\infty(\kappa_a)$ &
      time [s] \\
      \midrule
      $441  $  &  $1.83\cdot 10^{-9} $  &  $1.06$  &  $1.98\cdot 10^{-2}$  &  $1.54\cdot 10^{-4}$  &  $9.01\cdot 10^{-6}$  &  $ 1412.94$ \\
      $2601 $  &  $2.68\cdot 10^{-9} $  &  $1.88$  &  $1.76\cdot 10^{-2}$  &  $3.64\cdot 10^{-4}$  &  $1.80\cdot 10^{-5}$  &  $ 1865.92$ \\
      $6561 $  &  $3.71\cdot 10^{-10}$  &  $5.36$  &  $2.08\cdot 10^{-2}$  &  $5.55\cdot 10^{-4}$  &  $2.25\cdot 10^{-5}$  &  $ 4043.40$ \\
      $10201$  &  $1.22\cdot 10^{-9} $  &  $5.79$  &  $2.40\cdot 10^{-2}$  &  $8.64\cdot 10^{-4}$  &  $9.83\cdot 10^{-5}$  &  $12788.44$ \\
      \bottomrule
    \end{tabular}
    \caption{Relative appoximation errors and computation time for the
      approximation of the log-normal diffusion coefficient $\kappa$ for different
      numbers of FE degrees of freedom. The computation is done on uniform
      triangulations of $D$ with $M=20$ parameters and uses stochastic
      polynomials of degree smaller or equal than $10$ in each mode.  Here,
      $\kappa_{\mathrm{VMC}}$ is an approximation of $\kappa$ obtained via direct VMC and
      $\kappa_a$ is the output of Algorithm~\ref{alg:scale}.}
    \label{tab:experiments:logNormal:DoF_inc}
\end{table}

The computation time of our algorithm increases drastically as the number of FE DoFs get larger.
However, we suspect that this behavior originates from the discretization of the FE space.
As discussed in Section~\ref{sec:multivariate:TTformat}, the ranks of the operator $\discretized{W}$ and right-hand side $\discretized{b}$ depend quadratically on the ranks of the exponent $\gamma_\mathrm{VMC}$.
In our case the ranks are bounded by around $r=20$ for the first component tensor and the ranks decrease with the distance to the first component.
We also observe that the maximal ranks of $\gamma_{\mathrm{VMC}}$ increase as $M$ gets larger.
To improve storage capacity of the ALS algorithm, we \emph{round} $\discretized{W}$ to a precision of $10^{-12}$ by applying a truncated SVD to each component of the TT operator, which significantly reduces the ranks.
However, this process is computationally expensive as the finite element component of $\discretized{W}$ consists of a high-dimensional tensor whose sparsity is lost upon rounding.
This in turn increases storage capacity and computation time of the truncated SVD.
A different choice of spatial discretization by e.g.\ a reduced basis approach~\cite{Chen2013} could decrease the dimension of the deterministic approximation space and thus possibly reduce the computation time significantly.
That the computation time increases with the cardinality of the spatial discretization can also be observed in the next section on log-likelihood reconstruction.
A verification of this and possible improvements are subject to future work.

\subsection{Correlated Gaussian density approximation}%
\label{exptt:sec:experiments:correlated_gaussian_density_results}

A second field of possible applications of the discussed method is the reconstruction of probability densities or Bayesian likelihoods from approximations of their respective logarithms.
This section aims to investigate the performance of Algorithm~\ref{alg:scale} for the distribution of highly correlated and jointly Gaussian random variables, which often occur in practice when data are informative.
Denote by $J\in\mathbb{R}^{M\times M}$ the matrix of all ones and let $I\in\mathbb{R}^{M\times M}$ be the identity.
For some scale $\mu\in[0,1]$ we consider the covariance matrix $\Sigma_\mu = \mu I + (1-\mu)J$ and aim to reconstruct the probability density function of $\mathcal{N}(0,\Sigma_\mu)$, namely
\begin{equation*}
  \rho(y) = \frac{1}{(2\pi)^{M/2}\sqrt{\operatorname{det}(\Sigma_\mu)}} \exp\bigl( -\frac{1}{2} 
  y \cdot \Sigma_\mu^{-1} y
  \bigr).
\end{equation*}
We reconstruct $\rho$ for different choices of $\mu$ using Algorithm~\ref{alg:scale} without any additional scaling, i.e.\ $s=0$.
The experiments are conducted with $M=10$ parameters and we choose the quite large uniform polynomial degree $d_{\mathrm{a}}=60$ as approximation dimension since the highly correlated density is difficult to approximate on a tensor grid of univariate polynomials.
As $\log \rho$ is a quadratic polynomial in $y$,
we reconstruct an approximation of $\log\rho$ with the VMC method using $N_{\mathrm{VMC}}=10^{4}$ samples and a uniform polynomial degree of two, i.e.\ $d_h=3$.
The initial point for Algorithm~\ref{alg:scale} is chosen as $y_0=0\in\mathbb{R}^M$ and we use $\varepsilon=10^{-8}$ as a stopping criterion for the ALS algorithm.
To compare the results, we again compute an approximation of $\rho$ directly via VMC using the same number of training samples as in the reconstruction of $\log\rho$.

\begin{table}[htpb]%
    \centering
    \ra{1.1}  %
    \begin{tabular}{cccccccc}\toprule
      $\mu$
      & $\varepsilon_{\log\rho}^\infty(\log\rho_{\mathrm{VMC}})$
      & $r_{\mathrm{max}}(\rho_{\mathrm{a}})$
      & $\operatorname{res}(\rho_{\mathrm{a}})$
      & $\varepsilon_{\rho}^\infty(\rho_{\mathrm{a}})$
      & $r_{\mathrm{max}}(\rho_{\mathrm{VMC}})$
      & $\operatorname{res}(\rho_{\mathrm{VMC}})$
      & $\varepsilon_{\rho}^\infty(\rho_{\mathrm{VMC}})$\\
      \midrule
      $  1$ & $7.22\cdot 10^{-8}$ & $12$ & $1.05\cdot 10^{-7}$ & $5.53\cdot 10^{-7}$ & $1$ & $5.00\cdot 10^{-7}$ & $2.31\cdot 10^{-2}$\\
      $0.8$ & $9.00\cdot 10^{-8}$ & $12$ & $6.10\cdot 10^{-8}$ & $4.93\cdot 10^{-6}$ & $4$ & $1.42\cdot 10^{-6}$ & $7.10\cdot 10^{-2}$\\
      $0.6$ & $3.23\cdot 10^{-7}$ & $12$ & $8.71\cdot 10^{-8}$ & $1.85\cdot 10^{-3}$ & $4$ & $4.30\cdot 10^{-6}$ & $4.55\cdot 10^{-1}$\\
      $0.4$ & $6.34\cdot 10^{-7}$ & $12$ & $3.59\cdot 10^{-8}$ & $3.57\cdot 10^{-2}$ & $3$ & $3.13\cdot 10^{-5}$ & $6.49\cdot 10^{-1}$\\
      $0.2$ & $2.24\cdot 10^{-6}$ & $12$ & $4.83\cdot 10^{-7}$ & $1.08\cdot 10^{-1}$ & $2$ & $3.69\cdot 10^{-4}$ & $9.24\cdot 10^{-1}$\\
      \bottomrule
    \end{tabular}
    \caption{Maximum approximation rank $r_{\mathrm{max}}$, residual $operatorname{res}$ and relative $L^\infty$ errors for the VMC reconstruction of $\log\rho$, $\rho$ and the output of Algorithm~\ref{alg:scale} in relation to the correlation scale $\mu$.
             Smaller values of $\mu$ imply more correlated densities $\rho$.
             The computation is carried out for $M=10$ parameters with uniform stochastic dimension $3$ for $\log\rho$ and $61$ for the approximation of $\rho$.}%
    \label{tab:experiments:logLikelihood:correlated}
\end{table}

Table~\ref{tab:experiments:logLikelihood:correlated} shows the maximal ranks, the relative residuals and $L^\infty$-errors of the VMC reconstruction of $\log\rho$ and the approximation of $\rho$ obtained via Algorithm~\ref{alg:scale} and via the VMC method, for different values of $\mu$, respectively.
An increase in the correlation of the density seems to have no effect on the maximum rank of either of the reconstructed approximations.
However, as $\mu$ decreases and thus as the density becomes more and more correlated, we notice that the relative $L^\infty$ error of all approximations increases.
The approximation obtained via a direct VMC reconstruction has a maximum relative deviation from $\rho$ of approximately $2\%$ even for the uncorrelated case.
For the maximal correlation $\mu=0.2$ investigated, the VMC approach is only able to recover the constant zero function, resulting in a relative maximum discrepancy of almost $1$.
However, in comparison to the residual of the $\rho_\mathrm{a}$, the increase of $\varepsilon_{\rho}^\infty(\rho_{\mathrm{a}})$ is rather drastic and ranges from a relative error of $10^{-7}$ for a completely uncorrelated Gaussian density to a maximum relative deviation of $11\%$ for $\mu=0.2$.
Even though these results are far from optimal, we would like to point out that the reconstruction $\rho_{\mathrm{a}}$ of the proposed algorithm is orders of magnitudes better than what the direct sample-based reconstruction via VMC can achieve with the given number of samples.
We observe that the relative residual error $\operatorname{res}(\rho_\mathrm{a})$ remains constant in contrast to the increase of $\varepsilon_\rho^\infty(\rho_\mathrm{a})$.
On the other hand, $\operatorname{res}(\rho_\mathrm{VMC})$ increases in conjunction with $\varepsilon_\rho^\infty(\rho_\mathrm{VMC})$.
This verifies that no equivalence of the residual and the $L^\infty$-error can be assumed but that they might correlate for approximations not obtained via the (approximate) Galerkin projection of Algorithm~\ref{alg:scale}.
Moreover, we suspect that the large errors occurring for highly correlated densities are not intrinsic to the method presented in this work (nor the VMC method), but rather a hard problem for any tensorized approach due to the choice of a Cartesian coordinate system.
This problem can possibly be alleviated by a suitably chosen basis transform as e.g.\ proposed in~\cite{EGM20}.

\subsection{Bayesian likelihood approximation}%
\label{exptt:sec:experiments:logLikelihood_results}

Lastly, the following experiment examines the approximation quality of our approach for the Bayesian likelihood~\eqref{eq:experiments:bayesPotential}.
The forward map $\hat{G}(y) = w(y)\in H_0^1(D)$ is determined by the solution of the stationary diffusion problem~\eqref{exptt:eq:experiments:darcy} with log-normal random permeability $\kappa\in L^2(\mathbb{R}^M,\rho;L^\infty(D))$ which is specified by the affine exponent~\eqref{eq:experiments:KLexpansion}--\eqref{eq:experiments:KLexpansion:h_m} and constant right-hand side $f = 1$.
The parameter to observation map $G$ is the FE solution of $\hat{G}$ discretized with a lowest order conforming Lagrange finite element method with $J=2601$ degrees of freedom in the physical space and a maximal polynomial chaos degree of $2$ for all stochastic modes.
The observation
\begin{equation*}
    \delta = G(y^*) + \eta
\end{equation*}
is a perturbed realization of $G$ for some random sample $y^*\sim\mathcal{N}(0,I)$
where the perturbation noise $\eta$ is chosen with covariance $\Sigma = \sigma^2 I$ for $\sigma=10^{-3}$.
This introduces a relative measurement noise to the values of $G$ of about $5\%-10\%$.

\begin{table}[htpb]%
    \centering
    \ra{1.1}  %
    \begin{tabular}{cccccc}\toprule
      $M$ &
      $\mathcal{E}_G(G_{\mathrm{VMC}})$&
      $\mathcal{E}_\ell(\ell_{\mathrm{VMC}})$&
      $\Vert \discretized{B}\discretized{L_\mathrm{a}}-\discretized{f}\Vert_2$ &
      $\mathcal{E}_L(L_{a})$&
      $\mathcal{E}_{\hat{L}}(L_{a})$ \\
      \midrule
      $5 $  &  $2.80\cdot 10^{-4}$  &  $3.24\cdot 10^{-4}$  &  $8.46\cdot 10^{-7}$  &  $2.57\cdot 10^{-4}$  &  $2.79\cdot 10^{-6}$ \\
      $10$  &  $2.82\cdot 10^{-4}$  &  $4.30\cdot 10^{-4}$  &  $1.06\cdot 10^{-6}$  &  $2.69\cdot 10^{-4}$  &  $6.75\cdot 10^{-6}$ \\
      $20$  &  $3.04\cdot 10^{-4}$  &  $4.83\cdot 10^{-4}$  &  $8.12\cdot 10^{-7}$  &  $2.97\cdot 10^{-4}$  &  $4.97\cdot 10^{-6}$ \\
      $30$  &  $3.40\cdot 10^{-4}$  &  $3.53\cdot 10^{-4}$  &  $3.46\cdot 10^{-6}$  &  $3.68\cdot 10^{-4}$  &  $2.64\cdot 10^{-6}$ \\
      $40$  &  $3.49\cdot 10^{-4}$  &  $3.30\cdot 10^{-4}$  &  $1.67\cdot 10^{-5}$  &  $7.12\cdot 10^{-4}$  &  $3.14\cdot 10^{-5}$ \\
      \bottomrule
     \end{tabular}
     \caption{Absolute approximation errors for the
       approximation of the forward model $G$, the log-likelihood $\ell$
       and the likelihood $L$ for different expansion dimensions $M$.
       The forward model is discretized on a uniform triangulation with $5000$ triangles ($2601$ FE DoFs).
       Here, $\hat{L}=\exp \ell_{\mathrm{VMC}}$ is used as a reference
       for the error of our method.}%
     \label{tab:experiments:logLikelihood:abs}
 \end{table}

The absolute approximation errors for different quantities are depicted in Table~\ref{tab:experiments:logLikelihood:abs}.
Relative errors are shown in Table~\ref{tab:experiments:logLikelihood:rel}.
The VMC approximation of the forward map is denoted by $G_{\mathrm{VMC}}$ and with $\ell_{\mathrm{VMC}}$ we denote the TT representation of the log-likelihood that is computed algebraically from $G_{\mathrm{VMC}}$.
The likelihood approximations of our method is labeled by $L_a$ where the stochastic discretization space for each mode is restricted to polynomials of maximal degree $3$.
As stopping tolerance for our method, we set $\varepsilon=10^{-8}$.
Due to the relatively small function values of the log-likelihood, it suffices to set the scaling to $s=0$.
This implies $\tilde{d}_a = d_a = 4$ and renders the choice of $\varepsilon_s$ irrelevant.
As initial point for the method we choose $y_0=0\in\mathbb{R}^M$.
To determine if the approximation accuracy is limited by our method or by the reconstruction of the forward model $G$, we additionally compute the error between the likelihood approximation $L_a$ and samples $\hat L(y) = \exp \ell_{\mathrm{VMC}}(y)$ for $y\sim\mathcal{N}(0,I)$.
Finally, we compare the approximation obtained by our method to a merely sample based VMC tensor reconstruction $L_{\mathrm{VMC}}$ where the stochastic discretization space for each mode is restricted to polynomials of maximal degree $3$ as well.
We observe that $N_\mathrm{VMC}=10^3$ samples seem sufficient for the reconstructions and an increase of $N_\mathrm{VMC}$ yields no significant improvements.

\begin{table}[htpb]%
    \centering
    \ra{1.1}  %
    \begin{tabular}{ccccccc}\toprule
      $M$ &
      $\varepsilon_G(G_{\mathrm{VMC}})$&
      $\varepsilon_\ell(\ell_{\mathrm{VMC}})$&
      $\varepsilon_L(L_{\mathrm{VMC}})$&
      $\varepsilon_L(L_{a})$&
      $\varepsilon_{\hat{L}}(L_{a})$&
      time [s] \\
      \midrule
      $5 $  &  $1.86\cdot 10^{-2}$  &  $85.7$  &  $1.19\cdot 10^{-4}$  &  $2.57\cdot 10^{-4}$  &  $2.79\cdot 10^{-6}$  &  $0.09$ \\
      $10$  &  $1.87\cdot 10^{-2}$  &  $86.3$  &  $1.13\cdot 10^{-4}$  &  $2.69\cdot 10^{-4}$  &  $6.75\cdot 10^{-6}$  &  $0.78$ \\
      $20$  &  $2.02\cdot 10^{-2}$  &  $84.7$  &  $1.06\cdot 10^{-4}$  &  $2.97\cdot 10^{-4}$  &  $4.97\cdot 10^{-6}$  &  $0.29$ \\
      $30$  &  $2.26\cdot 10^{-2}$  &  $87.0$  &  $1.13\cdot 10^{-4}$  &  $3.68\cdot 10^{-4}$  &  $2.64\cdot 10^{-6}$  &  $0.77$ \\
      $40$  &  $2.32\cdot 10^{-2}$  &  $87.2$  &  $1.36\cdot 10^{-4}$  &  $7.12\cdot 10^{-4}$  &  $3.14\cdot 10^{-5}$  &  $1.67$ \\
      \bottomrule
    \end{tabular}
    \caption{Relative approximation errors for the
      approximation of the forward model $G$, the log-likelihood $\ell$
      and the likelihood $L$ for different expansion dimensions $M$.
      The forward model is discretized on a uniform triangulation with $5000$ triangles ($2601$ FE DoFs).
      Here, $\hat{L}=\exp \ell_{\mathrm{VMC}}$ is used as a reference
      for the error of our method. The last column is the measured time our algorithm
      requires to compute the likelihood $L_a$.}%
    \label{tab:experiments:logLikelihood:rel}
\end{table}

In Table~\ref{tab:experiments:logLikelihood:abs} it can be seen that the approximation of the forward model $G$ and the approximation of the log-likelihood $\ell$ seem to stagnate at an error of $10^{-4}$ independent of the number of modes.
Note that the error of the latter directly depends on the error of the former.
The absolute approximation quality of $L_a$ has the same order of magnitude as the one of $G_{\mathrm{VMC}}$, which is expected and also observed in e.g.~\cite{dolgov2019TTdensities}.
However, the last column of Table~\ref{tab:experiments:logLikelihood:abs} verifies that the main contribution of the approximation error originates from the error of the approximation of the forward map $G$ and not from Algorithm~\ref{alg:expTT}.
Here we assume the log-likelihood $\ell_{\mathrm{VMC}}$ to be exact and compute the error with respect to $\hat{L}=\exp\ell_{\mathrm{VMC}}$ instead of $L=\exp\ell$.

The stagnating error of $G_{\mathrm{VMC}}$ in Table~\ref{tab:experiments:logLikelihood:rel} indicates that VMC is not capable of recovering the solution of the forward problem with arbitrary accuracy without increasing the number of samples.
This entails a large approximation error for $\ell_{\mathrm{VMC}}$.
In conjunction with the fact that $\mathbb{E}[\vert\ell\vert] \ll 1$ this explains the large relative error ($>80$) of the log-likelihood approximation.
The relative error of the sample based tensor regression $L_{\mathrm{VMC}}$ is only slightly smaller than the approximation error of our method.
However, it is worth mentioning that the experiments suggest that our method only takes about one tenth of the computation time.
The last column of Table~\ref{tab:experiments:logLikelihood:rel} shows the measured time of our algorithm to compute the approximation $L_a$ from $\ell_{\mathrm{VMC}}$.
Similar to the approximation errors, the runtime seems to be rather independent from the number of expansion dimensions $M$ or increases only slightly as $M$ increases.
This is a very different behaviour when compared to the results of Section~\ref{exptt:sec:experiments:darcy_results}.
However, as mentioned before, the computation time seems to be correlated to (due to our ``naive'' implementation) the dimension of the deterministic component, which explains the short running times of our algorithm for the approximation of the real-valued likelihood $L$.

\section{Discussion}%
\label{sec:conclusion}

We derive a novel numerical approach to compute a low-rank approximation of the exponential of a multivariate function.
We assume that the exponent is given with respect to an orthonormal basis of (tensor product Hermite) polynomials and that the coefficient tensor of the expansion is in the tensor train (TT) format.
The central idea is to consider the exponential as the solution of a system of ordinary differential equations.
This allows us to approximate the function via a Galerkin projection method.
The Laplace-like structure of the corresponding operator and right-hand side allow for an efficient representation in the TT format, which renders the problem amenable to the ALS method.
We establish that the residual minimized by the ALS is equivalent to a certain energy norm up to a data oscillation term.
This not only implies that the ALS minimizes the distance to the exact solution in the energy norm but also that the resulting residual provides an error estimator for the solution, which in principle could be used for an adaptive refinement algorithm as well as in conjunction with other approximation algorithms, like cross approximation.

The algorithm is tested for the reconstruction of two different log-normal diffusion coefficients of a random elliptic PDE, a strongly correlated Gaussian density and a Bayesian likelihood, where the forward map is given by a  polynomial chaos surrogate in TT format.
We compare our results to established methods and to a black-box sample based reconstruction algorithm.
We observe that the performance of our approach is state-of-the-art with respect to the approximation accuracy, computation time and storage capacity for up to $M=40$ stochastic dimensions and large polynomial degrees.
Almost all computations are carried out on a common desktop computer\footnote{$2.1\,\mathrm{GHz}$ Intel Core i3 processor and $16\,\mathrm{GB}$ of memory.} with the exception of the experiments with large FE dimensions (last two rows in Table~\ref{tab:experiments:logNormal:DoF_inc}), for which slightly more memory was required to assemble the operator $\discretized{W}$.
However, it should be noted that the computation of an approximation of the log-normal diffusion field as investigated in Section~\ref{exptt:sec:experiments:darcy_results} can be quite expensive computationally for large FE dimensions.
In particular, our choice of spatial discretization in conjunction with our unoptimized implementation leads to long computations as the spatial dimension or the number of stochastic modes increase.
Nevertheless, we are confident that this can be alleviated by a more appropriate choice of discretization of the physical space and an optimized implementation for the assembly of the operator, which we leave as a subject for future works.
\footnote{The code implementing the numerical experiments is freely available at \url{https://bitbucket.org/trunschk/adaptive_vmc}.}

It should be emphasized that in principle the scope of our method reaches far beyond what is discussed and illustrated in this work.
The proposed method is applicable to a wide range of holonomic-like functions such as algebraic functions, sine and cosine, the error function, Bessel functions and hypergeometric functions.
Moreover, since the sought function may satisfy multiple differential equations it is possible to choose one for which the induced energy norm is best suited for the problem at hand.
This may also be obtained by considering other orthogonal bases, e.g.\ from the Askey scheme.
The dependence of the energy norm on the dynamical system clearly highlights a limitation of our current theory and it would be interesting to investigate for which classes of dynamical systems an equivalence of the energy norm to more convenient norms like the $L^2$ or $H^1$ norm can be established.

The exponential field approximation discussed here could be used to develop fully adaptive approximation schemes for the solutions of parametric PDEs similar to~\cite{EMPS20} with the crucial advantage that only pointwise evaluations of the solution of the considered PDE are required.
In contrast to the involved intrusive stochastic Galerkin methods of~\cite{EPS17,EMPS20}, a black-box adaptive non-intrusive method could be devised which still yields the Galerkin solution with high probability.

On a more practical side, it is possible to apply our algorithm to obtain a functional representations of a Bayesian posterior densities.
This allows, among other things, a very efficient computation of statistical quantities such as mean, variance, higher order moments and marginals (cf.~\cite{EGM20}) or fast generation of independent posterior samples (cf.~\cite{dolgov2019TTdensities}).
This is important in many reconstruction tasks such as~\cite{andrle2021invertible,Farchmin2020}.

\section*{Acknowledgements}
M.\ Eigel acknowledges the partial support of the DFG SPP 1886 ``Polymorphic Uncertainty Modelling for the Numerical Design of Structures''.
N.\ Farchmin has received funding form the German Central Innovation Program (ZIM) No. ZF4014017RR7.
P.\ Trunschke acknowledges support by the Berlin International Graduate School in Model and Simulation based Research (BIMoS).
The authors would like to thank Maren Casfor, Michael G\"otte, Robert Lasarzik, Mathias Oster, Leon Sallandt and Reinhold Schneider for fruitful discussions.

\printbibliography

\end{document}